\documentclass[reqno,11pt]{amsart}
\usepackage{amscd,amssymb,verbatim,array}
\usepackage{hyperref}
\usepackage{mathrsfs}

\numberwithin{equation}{section} \theoremstyle{plain}

\newcommand{\Complex}{\mathbb C}
\newcommand{\Real}{\mathbb R}

\newcommand{\ddbar}{\overline\partial}
\newcommand{\pr}{\partial}
\newcommand{\ol}{\overline}
\newcommand{\Td}{\widetilde}
\newcommand{\norm}[1]{\left\Vert#1\right\Vert}
\newcommand{\set}[1]{\left\{#1\right\}}
\newcommand{\To}{\rightarrow}
\DeclareMathOperator{\Ker}{Ker}

\newtheorem{theorem}{Theorem}[section]
\newtheorem{lemma}[theorem]{Lemma}
\newtheorem{proposition}[theorem]{Proposition}

\newtheorem{definition}[theorem]{Definition}

\theoremstyle{definition}

\theoremstyle{remark}

\numberwithin{equation}{section}

\newcommand{\abs}[1]{\lvert#1\rvert}

\begin{document}

\title[The anomaly formula of the analytic torsion on CR manifolds with $S^1$ action]
{The anomaly formula of the analytic torsion on CR manifolds with $S^1$ action}

\thanks{The author was supported by Taiwan Ministry of Science of Technology project 105-2115-M-008-008-MY2. The author also would like to express his gratitude to Dr. Chin-Yu Hsiao for very useful discussion in this work.}

\author{Rung-Tzung Huang}

\address{Department of Mathematics, National Central University, Chung-Li 320, Taiwan}

\email{rthuang@math.ncu.edu.tw}


\begin{abstract}
Let $X$ be a compact connected strongly pseudoconvex CR manifold of dimension $2n+1, n\ge 1$ with a transversal CR $S^1$-action on $X$. In this paper we introduce the Quillen metric on the determinant line of  the Fourier components of the Kohn-Rossi cohomology on $X$ with respect to the $S^1$-action. We study the behavior of the Quillen metric under the change of the metrics on the manifold $X$ and on the vector bundle over $X$. We obtain an anomaly formula for the Quillen metric on $X$ with respect to the $S^1$-action. 
\end{abstract}

\maketitle

\tableofcontents

\section{Introduction}

In \cite{RS2}, Ray and Singer introduced the holomorphic analytic torsion for $\overline{\partial}$-complex on complex manifolds as the complex analogue of the analytic torsion for flat vector bundles over Riemannian manfilds \cite{RS1}. Let $F$ be a Hermitian vector bundle over a compact Hermitian complex manifold $M$. Let $\lambda(F) = \otimes_q \left( \det H^q(M,F) \right)^{(-1)^{q+1}}$ be the dual of the determinant line of the Dolbeault cohomology groups of $M$ with values on $F$. In \cite{Qu}, Quillen defined a metric, the product of the $L^2$-metric on $\lambda(F)$ by the holomorphic analytic torsion, on $\lambda(F)$ when $M$ is a Riemann surface. In \cite{BGS3}, Bismut, Gillet and Soul\'{e} extended it to complex manifolds. By using probability method, they obtained the anomaly formulas for the Quillen metrics when the holomorphic bundle is endowed with Hermitian metrics and the base manifold is assumed to be K\"{a}hler. Recall that the anomaly formulas tell us the variation of the Quillen metrics with respect to the change of the Hermitian metrics on $TX$ and $F$. Note that, in \cite{B}, Berman considered high powers of a holomorphic line bundle over a complex manifold, where the metric of the base manifold is not necessarily K\"{a}hler, and obtained an asymptotic anomaly formula for the Quillen metric by using the Bergman type kernels.

In orbifold geometry, we have Kawasaki's Hirzebruch-Riemann-Roch formula~\cite{Ka79} and also general index theorem~\cite{PV}. Ma~\cite{M} first introduced analytic torsion on orbifolds and obtained anomaly and immersion formulas for Quillen metrics in the case of orbifolds,
which is expressed explicitly in the form of characteristic and secondary characteristic classes on orbifolds. Ma's results should play an important role toward establishing an arithmetic version of the Kawasaki-Riemann-Roch theorem in Arakelov geometry.

CR geometry is an important subject in several complex variables and is closely related to various research areas. To study further geometric problems for CR manifolds, it is important to know the corresponding heat kernel asymptotics and to have (local) index formula and the concept of analytic torsion. The difficulty comes from the fact that the Kohn Laplacian is not hypoelliptic. Thus, we should consider such problems on some class of CR manifolds. 
It turns out that Kohn's $\Box_b$ operator on CR manifolds with $S^1$ action including Sasakian manifolds of interest in String Theory (see~\cite{OV07}) is a natural one of geometric significance among those transversally elliptic operators initiated by Atiyah and Singer (see~\cite{Hsiao14},~\cite{HL1},~\cite{HsiaoLi15} and~\cite{CHT}). In \cite{HH}, Hsiao and the author considered a compact connected strongly pseudoconvex CR manifold $X$ and we introduced the Fourier components of the Ray-Singer analytic torsion on $X$ with respect to a transversal CR $S^1$-action. We established an asymptotic formula for the Fourier components of the analytic torsion with respect to the $S^1$-action. This generalizes the aymptotic formula of Bismut and Vasserot, \cite{BV}, on the holomorphic Ray-Singer torsion associated with high powers of a positive line bundle to strongly pseudoconvex CR manifolds with a transversal CR $S^1$-action. 

In a recent preprint, \cite{F}, Finski studied the general formula of the asymptotic expansion of Ray-Singer analytic torsion associated with increasing powers of a given positive line bundle and then the general asymptotic expansion of Ray-Singer analytic torsion for an orbifold and described a connection between the asymptotic formula of Ray-Singer analytic torsion for an orbifold in \cite{F} and our result in \cite{HH}. In another recent work, \cite{P0, P}, Puchol gave an asymptotic formula for the holomorphic analytic torsion forms of a fibration associated with increasing powers of a given positive line bundle which is the family version of the results of Bimsut and Vasserot on the asymptotic of the holomorphic torsion. 

In \cite{CM}, Cappell and Miller extended the holomorphic analytic torsion to coupling with an arbitrary holomorphic bundle with a compatible connection of type $(1, 1)$. They used certain not necessarily self-adjoint Laplacian to define the analytic torsion and, hence, the analytic torsion is complex-valued. In \cite{LY}, Liu and Yu established an explicit expression of the anomaly formula for the Cappell-Miller holomorphic torsion for K\"{a}hler manifolds by using heat kernel methods. In \cite{S}, Su proved an asymptotic formula for the Cappell-Miller holomorphic torsion associated with a high tensor power of a positive line bundle and a holomorphic vector bundle. 

In \cite{S0}, Su extended the holomorphic $L^2$ torsion introduced by Carey, Farber and Mathai in \cite{CFM} to the case without determinant class condition. He derived the anomaly formula for the holomorphic $L^2$ torsion under the change of the metrics. In the end, he studied the asymptotics of the holomorphic $L^2$ torsion associated with an increasing power of a positive line bundle.

In this paper we introduce the Quillen metric on the determinant line of the Fourier components of the Kohn-Rossi cohomology on $X$ with respect to a transversal CR $S^1$-action. We study the behavior of the Quillen metric under the change of the metrics on the manifold $X$ and on the vector bundles over $X$. We obtain an anomaly formula for the Quillen metric on $X$ with respect to the $S^1$-action, cf. Theorem \ref{T:main}, by using the heat kernel methods of \cite{CHT, HH, LY}.

\subsection{Motivation}
To motivate our approach, let's come back to complex geometry case. Let $M$ be a compact complex manifold of dimension $n$. Let $\langle \, \cdot \, , \, \cdot \, \rangle$ be a Hermitian metric on $\mathbb{C}TM$ and let $(F,h^F)\To M$ be a holomorpic vector bundle over $M$, where $h^F$ denotes a Hermitian fiber metric on $F$. Denote by $T^{*0,\bullet}M$ the vector bundle of $(0, \bullet)$ forms on $M$. Let $\Box_F$ be the Kodaira Laplacian with values in $T^{*0,\bullet}M\otimes F$ and $e^{-t \Box_F}$ be the associated heat operator. Denote by $\theta_{F}(z)$ the $\zeta$-function 
\[
\theta_{F}(z)  = - \mathcal{M} \left\lbrack \operatorname{STr}  \lbrack N e^{-t \Box_F} P^\perp  \rbrack \right\rbrack = -\operatorname{STr} \lbrack N(\Box_F)^{-z}  P^\perp \rbrack.
\] 
Here $N$ is the number operator on $T^{*0,\bullet}M$, $\operatorname{STr}$ denotes the super trace , $P^\perp$ is the orthogonal projection onto $({\rm Ker\,}\Box_F)^\perp$ and $\mathcal{M}$ denotes the Mellin transformation, cf. Definition~\ref{d-gue160313}. It is well-known that the $\zeta$-function has meromorphic extension to the whole complex plane. In particular, it is holomorphic at $z=0$. 


\begin{definition}
The analytic torsion associated to the holomorphic vector bundle $F$ over the complex manifold $M$ is defined by $\exp ( -\frac{1}{2} \theta_{F}'(0) )$. 
\end{definition}

For a finite dimensional vector space $V$, we set 
$$
\det V := \wedge^{\text{max}}V.
$$ 
We then denote by 
$$
(\det V)^{-1}:= (\det V)^*,
$$ 
the dual line of $\det V$. For $q=0,1,\cdots, n$, let $H^q(M,F)$ be the $q$-th $\overline{\partial}$-Dolbeault cohomology group with value in $F$. Denote by 
$$
H^\bullet(M, F) = \oplus_{q=0}^n H^q(M, F).
$$ Then
\[
\det H^\bullet(M,F) \, = \, \otimes_{q=0}^n \left( \det H^q(M,F)  \right)^{(-1)^q}
\]
is the determinant line of the Dolbeault cohomology $H^\bullet(M,F)$. We define 
\begin{equation}
\lambda(F) \, = \,  \left( \det H^\bullet(M,F)   \right)^{-1} \nonumber
\end{equation}
be the dual of $\det H^\bullet(M,F)$. By the Hodge theorem, the cohomology group $H^q(M,F)$ is isomorphic to the kernel of the Dolbeault Laplacian 
$$
\Box^{(q)}_{F} \, := \, \overline{\partial}^F\overline{\partial}^{F,*} + \overline{\partial}^{F,*}\overline{\partial}^F \,: \, \Omega^{0, q}(M, F) \to \Omega^{0, q}(M, F),
$$
where $\overline{\partial}^{F,*}$ denotes the adjoint of $\overline{\partial}^F$ with respect to the metrics $\langle \, \cdot \, , \, \cdot \, \rangle$ and $h^F$.
The metrics $\langle \, \cdot \, \, \cdot \, \rangle$ and $h^{F}$ induce a canonical $L^2$-metric $h^{H^\bullet(M,F)}$ on $H^\bullet(M,F)$. Let $|\cdot|_{\lambda(F)}$ be the $L^2$-metric on $\lambda(F)$ induced by $h^{H^\bullet(M,F)}$.

\begin{definition}\label{D:5.5.5-0}
The Quillen metric $\| \cdot  \|_{{\lambda}(F)}$ on $\det H^\bullet(M,F)$ is defined as
\[
\| \cdot  \|_{\lambda(F)} \, := \,  |\cdot|_{\lambda(F)} \cdot \exp ( -\frac{1}{2} \theta_{F}'(0)  ).
\]
\end{definition}

Now we recall the anomaly formula of Bismut, Gillet and Soul\'{e} for the Quillen metric on $\lambda(F)$. Let $\langle \, \cdot \, , \, \cdot \, \rangle'$ and ${h'}^F$ be another couple of  Hermtian metrics on $\mathbb{C}TM$ and on $F$, respectively. Let $\| \cdot  \|_{\lambda(F)}$ be the Quillen metric on $\lambda(F)$ associated to the metrics $\langle \, \cdot \, , \, \cdot \, \rangle$ and $h^F$ and let $\| \cdot  \|'_{\lambda(F)}$ be the Quillen metric on $\lambda(F)$ associated to the metrics $\langle \, \cdot \, , \, \cdot \, \rangle'$ and ${h'}^F$. Let $\nabla^{TM}$ and $\nabla'^{TM}$ be the Levit-Civita connections on $TM$ with respect to the metrics $\langle\, \cdot \, , \, \cdot \, \rangle$ and $\langle \, \cdot \, , \, \cdot \, \rangle'$ on $\mathbb{C}TM$, respectively. Let $P_{T^{1,0}M}$ be the natural projection from $\mathbb{C}TM$ onto $T^{1,0}M$. Then, 
$$
\nabla^{T^{1,0}M}\, := \, P_{T^{1,0}M}\nabla^{TM}
$$ 
and
$$
\nabla'^{T^{1,0}M}\, := \, P_{T^{1,0}M}\nabla'^{TM}
$$ 
are connections on $T^{1,0}M$. Let $\nabla^F$ and $\nabla'^F$ be the connections on $F$ induced by the Hermitian metrics $h^F$ and ${h'}^F$ on $F$, respectively. We denote by 
$$
\widetilde{\operatorname{Td}}(\nabla^{T^{1,0}M}, \nabla'^{T^{1,0}M}, T^{1,0}M) \quad \text{and} \quad 
\widetilde{\operatorname{ch}}(\nabla^{F}, \nabla'^{F}, F)
$$
the Bott-Chern classes, cf. \cite{BGS1}. We also denote by $\operatorname{Td}(\nabla'^{T^{1,0}M}, T^{1,0}M)$ the Todd class and $\operatorname{ch}(\nabla^F, F)$ the Chern character. We now assume that the metrics $\langle \, \cdot \, , \, \cdot \, \rangle$ and $\langle \, \cdot \, , \, \cdot \, \rangle'$ are K\"{a}hler. The anomaly formula of Bismut, Gillet and Soul\'{e} for Quillen metric on $\lambda(F)$, cf. \cite[Theorem 1.23]{BGS3}, is the following:
\begin{eqnarray}\label{E:BGSformula}
\log \left( \,  \frac{|| \cdot ||'_{\lambda(F)} }{|| \cdot ||_{\lambda(F)} }   \, \right)  \, = \,   \int_M \widetilde{\operatorname{Td}}(\nabla^{T^{1,0}M}, \nabla'^{T^{1,0}M}, T^{1,0}M) \wedge \operatorname{ch}(\nabla^F, F)  \nonumber \\
 +  \int_M \operatorname{Td}(\nabla'^{T^{1,0}M}, T^{1,0}M) \wedge \widetilde{\operatorname{ch}}(\nabla^{F}, \nabla'^{F}, F).
\end{eqnarray}

Let $(L,h^L)\To M$ be a holomorpic line bundle over $M$, where $h^L$ denotes a Hermitian fiber metric of $L$. Let $(L^*,h^{L^*})\To M$ be the dual bundle of $(L,h^L)$ and put 
$$
X=\set{v\in L^*;\, \abs{v}^2_{h^{L^*}}=1}.
$$ 
We call $X$ the circle bundle of $(L^*,h^{L^*})$. It is clear that $X$ is a compact CR manifold of dimension $2n+1$. Given a local holomorphic frame $s$ of $L$ on an open subset $U\subset M$, we define the associated local weight of $h^L$ by
$$
\abs{s(z)}^2_{h^L}=e^{-2\phi(z)}, \quad \phi\in C^\infty(U, \Real).
$$ 
The CR manifold $X$ is equipped with a natural  $S^1$ action. Locally, $X$ can be represented in local holomorphic coordinates $(z,\lambda)\in\mathbb C^{n+1}$, where $\lambda$ is the fiber coordinate, as the set of all $(z,\lambda)$ such that 
$$
\abs{\lambda}^2e^{2\phi(z)}=1,
$$
where $\phi$ is a local weight of $h^L$. The $S^1$ action on $X$ is given by 
$$
e^{i\theta}\circ (z,\lambda)=(z,e^{i\theta}\lambda), \quad e^{i\theta}\in S^1, \ (z,\lambda)\in X.
$$ 
Let $T\in C^\infty(X,TX)$ be the real vector field induced by the $S^1$ action, that is, 
$$
Tu=\frac{\pr}{\pr\theta}(u(e^{i\theta}\circ x))|_{\theta=0}, \quad u \in C^\infty(X).
$$ 
We can check that 
$$
[T,C^\infty(X,T^{1,0}X)]\subset C^\infty(X,T^{1,0}X)
$$
and 
$$
\Complex T(x)\oplus T^{1,0}_xX\oplus T^{0,1}_xX=\Complex T_xX
$$
(we say that the $S^1$ action is CR and transversal). For every $m\in\mathbb Z$, put 
\[\begin{split}
\Omega^{0,\bullet}_m(X):&=\set{u\in\Omega^{0,\bullet}(X);\, Tu=imu}\\
&=\set{u\in\Omega^{0,\bullet}(X);\, u(e^{i\theta}\circ x)=e^{im\theta}u(x), \forall\theta\in[0,2\pi[}.\end{split}\]
Since 
$$
\ddbar_bT=T\ddbar_b,
$$ 
we have 
$$
\ddbar_b:\Omega^{0,\bullet}_m(X)\To\Omega^{0,\bullet}_m(X),
$$ 
where $\ddbar_b$ denotes the tangential Cauchy-Riemann operator. Let $\Omega^{0,\bullet}(M,L^m)$ be the space of smooth sections of $(0,\bullet)$ forms of $M$ with values in $L^m$, where $L^m$ is the $m$-th power of $L$. It is known that (see Theorem 1.2 in~\cite{CHT}) there is a bijection 
\begin{equation}\label{E:bijec}
A_m:\Omega^{0,\bullet}_m(X)\To\Omega^{0,\bullet}(M,L^m)
\end{equation}
 such that 
 $$
 A_m\ddbar_{b}=\ddbar A_m
 $$ 
 on $\Omega^{0,\bullet}_m(X)$. Let $\Box_m$ be the Kodaira Laplacian with values in $T^{*0,\bullet}M\otimes L^m$ and let $e^{-t\Box_m}$ be the associated heat operator. It is well-known that $e^{-t\Box_m}$ admits an asymptotic expansion as $t\To0^+$. Consider 
 $$
 B_m(t):=(A_m)^{-1}\circ e^{-t\Box_m}\circ A_m.
 $$ 
 Let 
 \[\Box_{b,m}:\Omega^{0,\bullet}_m(X)\To\Omega^{0,\bullet}_m(X)\] 
 be the Kohn Laplacian for forms with values in the $m$-th $S^1$ Fourier component and let $e^{-t\Box_{b,m}}$ be the associated heat operator.  We can check that  
\begin{equation}\label{e-gue150923bi}
 e^{-t\Box_{b,m}}=B_m(t)\circ Q_m=Q_m\circ B_m(t)\circ Q_m,
\end{equation}
where 
$$
Q_m:\Omega^{0,\bullet}(X)\To\Omega^{0,\bullet}_m(X)
$$ 
is the orthogonal projection. From the asymptotic expansion of $e^{-t\Box_m}$ and \eqref{e-gue150923bi}, it is straightforward to see that
\begin{equation}\label{e-gue151108}
e^{-t\Box_{b,m}}(x,x)\sim t^{-n}a_n(x)+t^{-n+1}a_{n-1}(x)+\cdots.
\end{equation} 
From \eqref{e-gue151108}, we can define $\exp ( -\frac{1}{2} \theta_{b,m}'(0) )$ the $m$-th Fourier component of the analytic torsion on the CR manifold $X$, where 
$$
\theta_{b,m}(z)  = - \mathcal{M} \left\lbrack \operatorname{STr}  \lbrack N e^{-t \Box_{b,m}} \Pi^\perp_m  \rbrack \right\rbrack = -\operatorname{STr} \lbrack N(\Box_{b,m})^{-z}  P^\perp \rbrack.
$$ 
Here $N$ is the number operator on $T^{*0,\bullet}X$, $\operatorname{STr}$ denotes the super trace , $\Pi^\perp_m$ is the orthogonal projection onto $({\rm Ker\,}\Box_{b,m})^\perp$ and $\mathcal{M}$ denotes the Mellin transformation, cf. Definition~\ref{d-gue160313}. It is easy to see that 
\begin{equation}\label{e-gue160501}
\theta_{b,m}'(0)=\theta_{L^m}'(0). \nonumber
\end{equation}

For each $m \in \mathbb{Z}$ and $q=0,1, \cdots, n$, we consider the cohomology group:
\[
H^q_{b,m}(X) \, := \, \frac{\operatorname{Ker} \overline{\partial}_{b,m}:\Omega^{0,q}_m(X) \to \Omega^{0,q+1}_m(X)}{\operatorname{Im}\overline{\partial}_{b,m}:\Omega^{0,q-1}_m(X) \to \Omega^{0,q}_m(X)},
\]
and call it the $m$-th Fourier components of the Kohn-Rossi cohomology group. Recall that by \eqref{E:bijec} (see also \cite[Theorem 1.2]{CHT}), for each $m \in \mathbb{Z}$ and $q=0,1, \cdots, n$, the cohomology group $H^q_{b,m}(X)$ is isomorphic to the Dolbeault cohomology group $H^q(M,L^m)$. In particular, $\dim H^q_{b,m}(X)<\infty$. Denote by 
$$
H^\bullet_{b, m}(X) = \oplus_{q=0}^n H^q_{b,m}(X).
$$ 
Then
\[
\det H^\bullet_{b,m}(X) = \otimes_{q=0}^n \left( \det H^q_{b,m}(X)  \right)^{(-1)^q}
\]
is the determinant line of the cohomology $H^\bullet_{b,m}(X)$. We define 
\begin{equation}
\lambda_{b,m} \, = \, \left( \det H^\bullet_{b,m}(X)  \right)^{-1}. \nonumber
\end{equation}

Let $\langle\,\cdot\, | \,\cdot\,\rangle$ be the rigid Hermitian metric (see Definition~\ref{d-gue150514f}) on $\Complex TX$ given by, in local holomorphic coordinates $(z,\lambda)$, 
\[
\langle\,\frac{\pr}{\pr z_j}+i\frac{\pr\varphi}{\pr z_j}(z)\frac{\pr}{\pr\theta}\,|\,\frac{\pr}{\pr z_k}+i\frac{\pr\varphi}{\pr z_k}(z)\frac{\pr}{\pr\theta}\,\rangle \,  =  \, \langle\,\frac{\pr}{\pr z_j}\,,\,\frac{\pr}{\pr z_k}\,\rangle,\ \ j,k=1,2,\ldots,n.
\]
The metric $\langle \, \cdot \, |\, \cdot \, \rangle$ induces a canonical $L^2$-metric $h^{H^\bullet_{b,m}(X)}$ on $H^\bullet_{b,m}(X)$. Let $|\cdot|_{\lambda_{b,m}}$ be the $L^2$-metric on $\lambda_{b,m}$ induced by $h^{H^\bullet_{b,m}(X)}$. Fix $m \in \mathbb{Z}$. The Quillen metric $\| \cdot  \|_{\lambda_{b, m}}$ on $\det H^\bullet_{b,m}(X)$ is defined as
\[
\| \cdot  \|_{\lambda_{b, m}} \, := \,  |\cdot|_{\lambda_{b,m}} \cdot \exp ( -\frac{1}{2} \theta_{b,m}'(0)  ).
\]

 We now fix the Hermitian fiber metric $h^L$ on $L$ and, hence, the induced Hermitian metric $h^{L^m}$ on $L^m$ is also fixed. We assume that the metrics $\langle \, \cdot \, , \, \cdot \, \rangle$ and $\langle \, \cdot \, , \, \cdot \, \rangle'$ are K\"{a}hler. For the case of circle bundle over a compact complex manifold, the anomaly formula of Bismut, Gillet and Soul\'{e} for Quillen metric on $\lambda(L^m)$ over $M$ (see \eqref{E:BGSformula}) tells us:
\begin{eqnarray}\label{E:BGSformula1}
\log \left( \,  \frac{|| \cdot ||'_{\lambda(L^m)} }{|| \cdot ||_{\lambda(L^m)} }   \, \right)  \, = \,   \int_M \widetilde{\operatorname{Td}}(\nabla^{T^{1,0}M}, \nabla'^{T^{1,0}M}, T^{1,0}M) \wedge \operatorname{ch}(\nabla^{L^m}, L^m).  
\end{eqnarray}

Let $\nabla^{TX}$ and $\nabla'^{TX}$ be the Levit-Civita connections on $TX$ with respect to two different rigid Hermitian metrics $\langle\, \cdot \, | \, \cdot \, \rangle$ and $\langle \, \cdot \, | \, \cdot \, \rangle'$ on $\mathbb{C}TX$, respectively. Let $P_{T^{1,0}X}$ be the natural projection from $\mathbb{C}TX$ onto $T^{1,0}X$. Then, 
$$
\nabla^{T^{1,0}X}\, := \, P_{T^{1,0}X}\nabla^{TX} 
$$
and
$$
\quad \nabla'^{T^{1,0}X}\, := \, P_{T^{1,0}X}\nabla'^{TX}
$$ 
are connections on $T^{1,0}X$. We denote by 
$
\widetilde{\operatorname{Td}}_{b}(\nabla^{T^{1,0}X}, \nabla'^{T^{1,0}X}, T^{1,0}X)$
the tangential Bott-Chern class, cf. Subsection \ref{S:ts}.
We denote by $|| \cdot ||_{\lambda_{b,m}}$ and $|| \cdot ||_{\lambda_{b,m}}'$ the Quillen metrics on $\det H^\bullet_{b,m}(X)$ with respect to the rigid Hermitian metrics $\langle\, \cdot \, | \, \cdot \, \rangle$ and $\langle \, \cdot \, | \, \cdot \, \rangle'$, respectively. We can now reformulate \eqref{E:BGSformula1} in terms of geometric objects on $X$:
\begin{eqnarray}\label{E:main}
\log \left( \,  \frac{|| \cdot ||'_{\lambda_{b,m}} }{|| \cdot ||_{\lambda_{b,m}} }   \, \right)  \, = \, \frac{1}{2\pi}  \int_X \widetilde{\operatorname{Td}}_{b}(\nabla^{T^{1,0}X}, \nabla'^{T^{1,0}X}, T^{1,0}X) \wedge e^{-m\frac{d\omega_0}{2\pi}} \wedge \omega_0, \nonumber
\end{eqnarray}
where $e^{-m\frac{d\omega_0}{2\pi}}$ denotes the Chern polynomial of the Levi curvature, cf. \eqref{E:leviform2}, and $\omega_0$ is the unique one form given by \eqref{E:global1form}.

The purpose of this paper is to establish the anomaly formula on any abstract strongly pseudoconvex CR manifolds with a transversal CR locally free $S^1$-action. Note that for the case of circle bundle, the $S^1$ action is globally free and $X$ is strongly pseudoconvex if $L$ is positve. 


\subsection{Main result}\label{s-gue150508a}

We now formulate the main results. We refer to Section~\ref{s-gue150508bI} for some notations and terminology used here. 

Let $(X, T^{1,0}X)$ be a compact connected strongly pseudoconvex CR manifold with a transversal CR locally free $S^1$ action $e^{i\theta}$ (see Definition~\ref{d-gue160502}), where $T^{1,0}X$ is a CR structure of $X$. 
Let $T\in C^\infty(X,TX)$ be the real vector field induced by the $S^1$ action and let $\omega_0\in C^\infty(X,T^*X)$ be the global real one form determined by 
\begin{equation}\label{e-gue160502b}
\langle\,\omega_0\,,\,T\,\rangle=-1,\ \ \langle\,\omega_0\,,\,u\,\rangle=0,\ \ \forall u\in T^{1,0}X\oplus T^{0,1}X. \nonumber
\end{equation} 
For $x\in X$, we say that the period of $x$ is $\frac{2\pi}{\ell}$, $\ell\in\mathbb N$, if $e^{i\theta}\circ x\neq x$, for every $0<\theta<\frac{2\pi}{\ell}$ and $e^{i\frac{2\pi}{\ell}}\circ x=x$. For each $\ell\in\mathbb N$, put 
\begin{equation}\label{e-gue150802bm}
X_\ell=\set{x\in X;\, \mbox{the period of $x$ is $\frac{2\pi}{\ell}$}} 
\end{equation} 
and let 
$$
p=\min\set{\ell\in\mathbb N;\, X_\ell\neq\emptyset}
$$ 
It is well-known that if $X$ is connected, then $X_p$ is an open and dense subset of $X$ (see Duistermaat-Heckman~\cite{Du82}). 
In this work, we assume that $p=1$ 
and we denote 
$$
X_{{\rm reg\,}}:=X_{p}=X_1.
$$ 
We call $x\in X_{{\rm reg\,}}$ a regular point of the $S^1$ action. Let $X_{{\rm sing\,}}$ be the complement of $X_{{\rm reg\,}}$. 

Let $E$ be a rigid CR vector bundle over $X$ (see Definition~\ref{d-gue150508dI}) and we take a rigid Hermitian metric $\langle\,\cdot\,|\,\cdot\,\rangle_E$ on $E$ (see Definition~\ref{d-gue150514f}). Take a rigid Hermitian metric $\langle\,\cdot\,|\,\cdot\,\rangle$ on $\Complex TX$ such that 
$$
T^{1,0}X\perp T^{0,1}X, \quad T\perp (T^{1,0}X\oplus T^{0,1}X), \quad \langle\,T\,|\,T\,\rangle=1
$$
and let $\langle\,\cdot\,|\,\cdot\,\rangle_E$ be the Hermitian metric on $T^{*0,\bullet}X\otimes E$ induced by the fixed Hermitian metrics on $E$ and $\Complex TX$. 
We denote by $dv_X=dv_X(x)$ the volume form on $X$ induced by the Hermitian metric $\langle\,\cdot\,|\,\cdot\,\rangle$ on $\Complex TX$. Then we get natural global $L^2$ inner product $(\,\cdot\,|\,\cdot\,)_{E}$ on $\Omega^{0,\bullet}(X,E)$. We denote by $L^{2}(X,T^{*0,\bullet}X\otimes E)$ the completion of $\Omega^{0,\bullet}(X,E)$ with respect to $(\,\cdot\,|\,\cdot\,)_{E}$. For every $u\in\Omega^{0,\bullet}(X,E)$, we can define $Tu\in\Omega^{0,\bullet}(X,E)$ and we have $T\ddbar_b=\ddbar_bT$.
For $m\in\mathbb Z$, put 
\[\begin{split}
\Omega^{0,\bullet}_m(X,E):&=\set{u\in\Omega^{0,\bullet}(X,E);\, Tu=imu}\\
&=\set{u\in\Omega^{0,\bullet}(X,E);\, (e^{i\theta})^*u=e^{im\theta}u,\ \ \forall\theta\in[0,2\pi[},\end{split}\]
where $(e^{i\theta})^*$ denotes the pull-back map by $e^{i\theta}$ (see \eqref{e-gue150508faI}). 
For each $m\in\mathbb Z$, we denote by $L^{2}_m(X,T^{*0,\bullet}X\otimes E)$ the completion of $\Omega^{0,\bullet}_m(X,E)$ with respect to $(\,\cdot\,|\,\cdot\,)_{E}$. 

Since 
$$
T\ddbar_b=\ddbar_bT,
$$ 
we have 
\[\ddbar_{b,m}:=\ddbar_b:\Omega^{0,\bullet}_m(X,E)\To\Omega^{0,\bullet}_m(X,E).\] 
We also write
\[\ol{\pr}^{*}_b:\Omega^{0,\bullet}(X,E)\To\Omega^{0,\bullet}(X,E)\]
to denote the formal adjoint of $\ddbar_b$ with respect to $(\,\cdot\,|\,\cdot\,)_E$. Since $\langle\,\cdot\,|\,\cdot\,\rangle_E$ and $\langle\,\cdot\,|\,\cdot\,\rangle$ are rigid, we can check that 
\begin{equation}\label{e-gue150517im}
\begin{split}
&T\ddbar^{*}_b=\ddbar^{*}_bT\ \ \mbox{on $\Omega^{0,\bullet}(X,E)$},\\
&\ddbar^{*}_{b,m}:=\ddbar^{*}_b:\Omega^{0,\bullet}_m(X,E)\To\Omega^{0,\bullet}_m(X,E),\ \ \forall m\in\mathbb Z.
\end{split}
\end{equation}
Let $\Box_{b,m}$ denote the $m$-th Kohn Laplacian given by
\begin{equation}\label{e-gue151113ym}
\Box_{b,m}:=(\ddbar_b+\ddbar^*_b)^2:\Omega^{0,\bullet}_m(X,E)\To\Omega^{0,\bullet}_m(X,E).
\end{equation}
We extend $\Box_{b,m}$ to $L^{2}_m(X,T^{*0,\bullet}X\otimes E)$ by 
\begin{equation}\label{e-gue151113ymI}
\Box_{b,m}:{\rm Dom\,}\Box_{b,m}\subset L^{2}_m(X,T^{*0,\bullet}X\otimes E)\To L^{2}_m(X,T^{*0,\bullet}X\otimes E)\,,
\end{equation}
where 
$$
{\rm Dom\,}\Box_{b,m}:=\{u\in L^{2}_m(X,T^{*0,\bullet}X\otimes E);\, \Box_{b,m}u\in L^{2}_m(X,T^{*0,\bullet}X\otimes E)\},
$$ 
for which, for any $u\in L^{2}_m(X,T^{*0,\bullet}X\otimes E)$, $\Box_{b,m}u$ is defined in the sense of distribution. It is known that $\Box_{b,m}$ is self-adjoint, ${\rm Spec\,}\Box_{b,m}$ is a discrete subset of $[0,\infty[$ and for every $\nu\in{\rm Spec\,}\Box_{b,m}$, $\nu$ is an eigenvalue of $\Box_{b,m}$ (see Section 3 in~\cite{CHT}). Let 
$
e^{-t\Box_{b,m}} 
$
be associated heat operator. 
Let $N$ be the number operator on $T^{*0,\bullet}X$, i.e. $N$ acts on $T^{*0,q}X$ by multiplication by $q$, and $\operatorname{STr}$ denotes the super trace (see the discussion in the beginning of Section~\ref{s-gue160502q}). We denote by 
$$
\Pi^\perp_m:L^2_m(X,T^{0,\bullet}X\otimes E)\To({\rm Ker\,}\Box_{b,m})^\perp
$$ 
the orthogonal projection. From \eqref{E:5.5.10b}, for $\operatorname{Re}(z)>n$, we can define the $\zeta$ function
\begin{equation}\label{E:5.5.12b}
\theta_{b,m}(z)  = - \mathcal{M} \left\lbrack \operatorname{STr}  \lbrack N e^{-t \Box_{b,m}} \Pi^\perp_m  \rbrack \right\rbrack  =   - \operatorname{STr} \left\lbrack N ({\Box}_{b,m})^{-z} {\Pi}^\perp_m \right\rbrack \nonumber
\end{equation}
and $\theta_{b,m}(z)$ extends to a meromorphic function on $\mathbb{C}$ with poles contained in the set  
\[\set{\ell-\frac{j}{2};\, \ell,j\in\mathbb Z},\]
its possible poles are simple, and $\theta_{b,m}(z)$ is holomorphic at $0$ (see Lemma \ref{L:meroext} or \cite[Lemma 4.4]{HH}), where $\mathcal{M}$ denotes the Mellin transformation, cf. Definition~\ref{d-gue160313}. The $m$-th Fourier component of the analytic torsion for the vector bundle $E$ over $X$ is given by $\exp ( -\frac{1}{2} \theta_{b,m}'(0) )$ (see Definition~\ref{d-gue160502w}). 

Denote by 
$$
H^\bullet_{b, m}(X, E) = \oplus_{q=0}^n H^q_{b,m}(X,E),
$$ 
where $H^q_{b,m}(X,E), \ q=0,1, \cdots, n,$ is the $m$-th Fourier components of the Kohn-Rossi cohomology group (see Definition \ref{D:krcohgp}).
Then
\[
\det H^\bullet_{b,m}(X,E) = \otimes_{q=0}^n \left( \det H^q_{b,m}(X,E)  \right)^{(-1)^q}
\]
is the determinant line of the cohomology $H^\bullet_{b,m}(X,E)$. We define 
\begin{equation}
\lambda_{b,m}(E) = \left( \det H^\bullet_{b,m}(X,E)  \right)^{-1}. \nonumber
\end{equation}

By Theorem 3.7 of \cite{CHT}, the cohomology $H^q_{b,m}(X,E)$ is isomorphic to the kernel of $\Box^{(q)}_{b,m}$. The metrics $\langle \, \cdot \, |\, \cdot \, \rangle$ and $\langle \, \cdot \, |\, \cdot \, \rangle_E$ induce a canonical $L^2$-metric $h^{H^\bullet_{b,m}(X,E)}$ on $H^\bullet_{b,m}(X,E)$. Let $|\cdot|_{\lambda_{b,m}(E)}$ be the $L^2$-metric on $\lambda_{b,m}(E)$ induced by $h^{H^\bullet_{b,m}(X,E)}$. Fix $m \in \mathbb{Z}$. The Quillen metric $\| \cdot  \|_{\lambda_{b, m}(E)}$ on $\det H^\bullet_{b,m}(X,E)$ is defined as
\[
\| \cdot  \|_{\lambda_{b, m}(E)} \, := \,  |\cdot|_{\lambda_{b,m}(E)} \cdot \exp ( -\frac{1}{2} \theta_{b,m}'(0)  ).
\]

Let $\nabla^{TX}$ and $\nabla'^{TX}$ be the Levit-Civita connections on $TX$ with respect to the rigid Hermitian metrics $\langle\, \cdot \, | \, \cdot \, \rangle$ and $\langle \, \cdot \, | \, \cdot \, \rangle'$ on $\mathbb{C}TX$, respectively. Let $P_{T^{1,0}X}$ be the natural projection from $\mathbb{C}TX$ onto $T^{1,0}X$. Then, 
$$
\nabla^{T^{1,0}X}\, := \, P_{T^{1,0}X}\nabla^{TX}
$$
and 
$$
\nabla'^{T^{1,0}X}\, := \, P_{T^{1,0}X}\nabla'^{TX}
$$ 
are connections on $T^{1,0}X$. Let $\nabla^E$ and $\nabla'^E$ be the connections on $E$ induced by the rigid Hermitian metrics $h^E$ and ${h'}^E$ on $E$, respectively.
Denote by $\widetilde{\operatorname{Td}}_{b}(\nabla^{T^{1,0}X}, \nabla'^{T^{1,0}X}, T^{1,0}X)$ and $\widetilde{\operatorname{ch}}_{b}(\nabla^{E}, \nabla'^{E}, E)$ the tangential Bott-Chern classes,
$\operatorname{ch}_b(\nabla^E, E)$ the tangential Chern character and $\operatorname{Td}_{b}(\nabla'^{T^{1,0}X}, T^{1,0}X)$ the tangential Todd class, cf. Subsection \ref{S:ts}.

Our main result is the following

\begin{theorem}\label{T:main0}
With the notations and assumptions above, the following identity holds:
\begin{eqnarray}
\log \left( \,  \frac{|| \cdot ||'_{\lambda_{b,m}(E)} }{|| \cdot ||_{\lambda_{b,m}(E)} }   \, \right)  \, = \, \frac{1}{2\pi}  \int_X \widetilde{\operatorname{Td}}_{b}(\nabla^{T^{1,0}X}, \nabla'^{T^{1,0}X}, T^{1,0}X) \wedge \operatorname{ch}_b(\nabla^E, E) \wedge e^{-m\frac{d\omega_0}{2\pi}} \wedge \omega_0 \nonumber \\
 + \frac{1}{2\pi} \int_X \operatorname{Td}_{b}(\nabla'^{T^{1,0}X}, T^{1,0}X) \wedge \widetilde{\operatorname{ch}}_{b}(\nabla^{E}, \nabla'^{E}, E) \wedge e^{-m \frac{d\omega_0}{2\pi}} \wedge \omega_0, \nonumber
\end{eqnarray}
where $e^{-m\frac{d\omega_0}{2\pi}}$ denotes the Chern polynomial of the Levi curvature, cf. \eqref{E:leviform2}, and $\omega_0$ is the unique one form given by \eqref{e-gue150802bm}, see also \eqref{E:global1form}.
\end{theorem}

Note that the proof of Theorem~\ref{T:main0} is based on Theorem~\ref{T:5.5.6}, Theorem~\ref{t-gue150607}, Theorem~\ref{t-gue150630I} and Theorem~\ref{T:a0zz} which are the main technical results of this paper.

This paper is organized as follows. In Section 2, we collect some notations, definitions and terminology we use throughout and state our main result. In the end of this section, we deduce our anomaly formula on some class of orbifold line bundle. In Section 3, we study the asymptotic behavior of certain heat kernels when $t \to 0^+$. In Section 4, we establish the anomaly formula for the $m$-th Fourier components of the Quillen metric on CR manifolds with a transversal CR $S^1$-action. In Section 5, we establish an asymptotic anomaly formula for the $m$-th Fourier component of the Quillen metric on CR manifolds with a transversal CR $S^1$-action. 


\section{Preliminaries and statement of main result}
In Subsection \ref{s-gue150508bI}, we collect some notations, definitions and terminology we use throughout. In Subsection \ref{S:hkokl}, we recall some background on heat kernels of Kohn Laplacian. In Subsection \ref{S:mtransf}, we recall the definition of Melin transformation. In Subsection \ref{s-gue160502q}, we recall the definition of the Fourier components of the analytic torsion and define the Quillen metric. In Subsection \ref{S:ts}, we define the tangential characteristic and Bott-Chern classes. In Subsection \ref{S:mainth}, we state our main result. Finally, in Subsection \ref{S:orbi}, we deduce our anomaly formula on some class of orbifold line bundle.

\subsection{Set up and terminology}\label{s-gue150508bI} 

Let $(X, T^{1,0}X)$ be a compact CR manifold of dimension $2n+1$, $n\geq 1$, where $T^{1,0}X$ is a CR structure of $X$, that is, $T^{1,0}X$ is a subbundle of the complexified tangent bundle $\mathbb{C}TX$ of rank $n$ satisfying 
$$
T^{1,0}X\cap T^{0,1}X=\{0\},
$$ 
where 
$$
T^{0,1}X=\overline{T^{1,0}X} \quad
 \text{and} 
\quad
[\mathcal V,\mathcal V]\subset\mathcal V,
$$ where $\mathcal V=C^\infty(X, T^{1,0}X)$. There is a unique subbundle $HX$ of $TX$ such that 
$$
\mathbb{C}HX=T^{1,0}X \oplus T^{0,1}X,
$$
 i.e. $HX$ is the real part of $T^{1,0}X \oplus T^{0,1}X$. Let $J \, : \, HX \to HX$ be the complex structure map given by 
$$
J(u+\bar{u}) \, = \, iu-i\bar{u},
$$ 
for every $u \, \in \, T^{1,0}X$. By complex linear extension $J$ to $\mathbb{C}TX$, the $i$-eigenspace of $J$ is given by 
$$
T^{1,0}X \, = \, \left\{ V \in \mathbb{C}HX \, : \, JV \, = \, \sqrt{-1}V  \right\}.
$$ 
We shall also write $(X, HX, J)$ to denote a compact CR manifold. Let $E$ be a smooth vector bundle over $X$. We use $\Gamma(E)$ to denote the space of smooth sections of $E$ on $X$.

Let $(X, HX, J)$ be a compact CR manifold. From now on, we assume that $(X, HX, J)$ admits a $S^1$ action: 
$$
S^1\times X\rightarrow X, (e^{i\theta}, x) \mapsto e^{i\theta}\circ x.
$$ 
We write $e^{i\theta}$ to denote the $S^1$ action. Let $T\in C^\infty(X, TX)$ be the global real vector field induced by the $S^1$ action given by 
$$
(Tu)(x) \, = \, 
\frac{\partial}{\partial\theta}\left(u(e^{i\theta}\circ x)\right)|_{\theta=0},  \quad u\in C^\infty(X).
$$ 

\begin{definition}\label{d-gue160502}
We say that the $S^1$ action $e^{i\theta}$ is CR if
$$[T, C^\infty(X, T^{1,0}X)]\subset C^\infty(X, T^{1,0}X)$$ and the $S^1$ action is transversal if, for each $x\in X$,
$$\mathbb CT_xX=T_x^{1,0}X\oplus T_x^{0,1}X\oplus\Complex T(x).$$ Moreover, we say that the $S^1$ action is locally free if $T\neq0$ everywhere. It should be mentioned that transversality implies locally free.
\end{definition}

We assume throughout that $(X, T^{1,0}X)$ is a compact connected CR manifold with a transversal CR locally free $S^1$ action $e^{i\theta}$ and we let $T$ be the global vector field induced by the $S^1$ action. 
Then $L_TJ=0$ on $HX$, where $L_T$ denotes the Lie derivative along the direction $T$, cf. \cite[Lemma 2.3]{HLM}.

Since 
$$
[\Gamma(T^{1,0}X), \Gamma(T^{1,0}X)] \subset \Gamma(T^{1,0}X),
$$ 
we have 
$$
[JU, JV] - [U, V] \in C^\infty(X, HX),
$$ 
for all $U, V \in C^\infty(X, HX)$. Let $\omega_0\in C^\infty(X,T^*X)$ be the global real one form dual to $T$, that is,
\begin{equation}\label{E:global1form}
\langle\,\omega_0\,,\,T\,\rangle=-1, \quad \langle\,\omega_0\,,\,HX\,\rangle=0. 
\end{equation}
Then, for each $x \in X$, we define a quadratic form on $HX$ by
\begin{equation}\label{E:leviform1}
\mathcal{L}_x(U, V) \, = \, \frac{1}{2} d\omega_0(JU, V), \quad \forall \, U, V\in H_xX.  \nonumber
\end{equation}
We extend $\mathcal{L}$ to $\mathbb{C}HX$ by complex linear extension. Then, for $U, V \in T_x^{1,0}X,$
\begin{equation}\label{E:leviform2}
\mathcal{L}_x(U,\ol V) \, = \, \frac{1}{2} d\omega_0(JU, V) \, = \, -\frac{1}{2i} d\omega_0(U, \ol V).
\end{equation}
The Hermitian quadratic form $\mathcal{L}_x$ on $T_x^{1,0}X$ is called the Levi form at $x$.

\begin{definition}
We say that $T^{1,0}X$ is a strongly pseudoconvex structure and $X$ is a strongly pseudoconvex CR manifold if the Levi form $\mathcal{L}_x$ is a positive definite quadratic form on $H_xX$, for each $x \in X$.
\end{definition}

We further assume throughout that $(X, T^{1,0}X)$ is a compact connected strongly pseudoconvex CR manifold with a transversally CR locally free $S^1$-action. It should be noted that a strongly pseudoconvex CR manifold is always a contact manifold. From \eqref{E:global1form}, we see that $\omega_0$ is a contact form, $HX$ is the contact plane and $T$ is the Reeb vector field.

Denote by $T^{*1,0}X$ and $T^{*0,1}X$ the dual bundles of
$T^{1,0}X$ and $T^{0,1}X$, respectively. Define the vector bundle of $(p,q)$ forms by
$$
T^{*p,q}X \, := \, \Lambda^p(T^{*1,0}X) \wedge \Lambda^q(T^{*0,1}X).
$$ 
Put 
$$
T^{*0,\bullet}X \, := \, \oplus_{j\in\set{0,1,\ldots,n}}T^{*0,j}X.
$$
Let $D \subset X$ be an open subset. Let $\Omega^{p,q}(D)$
denote the space of smooth sections of $T^{*p,q}X$ over $D$ and let $\Omega_0^{p,q}(D)$
be the subspace of $\Omega^{p,q}(D)$ whose elements have compact support in $D$. Put 
\[\begin{split}
&\Omega^{0,\bullet}(D) \, := \, \oplus_{j\in\set{0,1,\ldots,n}}\Omega^{0,j}(D),\\
&\Omega^{0,\bullet}_0(D) \, := \, \oplus_{j\in\set{0,1,\ldots,n}}\Omega^{0,j}_0(D).\end{split}\]
Similarly, if $E$ is a vector bundle over $D$, then we let $\Omega^{p,q}(D,E)$ denote the space of smooth sections of $T^{*p,q}X\otimes E$ over $D$ and let $\Omega_0^{p,q}(D,E)$ be the subspace of $\Omega^{p,q}(D, E)$ whose elements have compact support in $D$. Put 
 \[\begin{split}
&\Omega^{0,\bullet}(D,E) \, := \, \oplus_{j\in\set{0,1,\ldots,n}}\Omega^{0,j}(D,E),\\
&\Omega^{0,\bullet}_0(D,E) \, := \, \oplus_{j\in\set{0,1,\ldots,n}}\Omega^{0,j}_0(D,E).\end{split}\]

Fix $\theta_0\in]-\pi, \pi[$, $\theta_0$ small. Let
$$d e^{i\theta_0}: \mathbb CT_x X\rightarrow \mathbb CT_{e^{i\theta_0}x}X$$
denote the differential map of $e^{i\theta_0}: X\rightarrow X$. By the CR property of the $S^1$ action, we can check that
\begin{equation}\label{e-gue150508fa}
\begin{split}
de^{i\theta_0} \, : \, T_x^{1,0}X\rightarrow T^{1,0}_{e^{i\theta_0}x}X,\\
de^{i\theta_0} \, : \, T_x^{0,1}X\rightarrow T^{0,1}_{e^{i\theta_0}x}X,\\
de^{i\theta_0}(T(x)) \, = \, T(e^{i\theta_0}x).
\end{split}
\end{equation}
Let 
$$
(de^{i\theta_0})^* \, : \, \Lambda^r(\Complex T^*X)\To\Lambda^r(\Complex T^*X)
$$ 
be the pull-back map by $e^{i\theta_0}$, $r=0,1,\ldots,2n+1$. From \eqref{e-gue150508fa}, it is easy to see that, for every $q=0,1,\ldots,n$, 
\begin{equation}\label{e-gue150508faI}
(de^{i\theta_0})^* \, : \, T^{*0,q}_{e^{i\theta_0}\circ x}X\To T^{*0,q}_{x}X.
\end{equation}

For $u\in\Omega^{0,q}(X)$, we define $Tu$ as follows:
\begin{equation}\label{e-gue150508faII}
(Tu)(X_1, \cdots, X_q) \, :=\, \frac{\pr}{\pr\theta}\Bigr((de^{i\theta})^*u(X_1, \cdots, X_q)\Bigr)\Big|_{\theta=0}, \quad X_1, \cdots, X_q \in T^{0,1}_xX.
\end{equation}
From \eqref{e-gue150508faI} and \eqref{e-gue150508faII}, we have 
$$
Tu \in \Omega^{0,q}(X),
$$ 
for all $u \in \Omega^{0,q}(X)$. From the definition of $Tu$, it is easy to check that 
$$
Tu=L_Tu,
$$ 
for $u \in \Omega^{0,q}(X)$, where $L_Tu$ is the Lie derivative of $u$ along the direction $T$. For every $\theta\in\Real$ and every $u\in C^\infty(X,\Lambda^r(\Complex T^*X))$, we write 
$$
u(e^{i\theta}\circ x):=(de^{i\theta})^*u(x).
$$ 
It is clear that, for every $u\in C^\infty(X,\Lambda^r(\Complex T^*X))$, we have 
\begin{equation}\label{e-gue150510f}
u(x) \, = \, \sum_{m\in\mathbb Z}\frac{1}{2\pi}\int^{\pi}_{-\pi}u(e^{i\theta}\circ x)e^{-im\theta}d\theta. \nonumber
\end{equation}

Let 
$$
\ddbar_b:\Omega^{0,q}(X)\rightarrow\Omega^{0,q+1}(X)
$$ 
be the Cauchy-Riemann operator. From the CR property of the $S^1$ action, it is straightforward from \eqref{e-gue150508faI} and \eqref{e-gue150508faII} to see that
\[T\ddbar_b=\ddbar_bT\ \ \mbox{on $\Omega^{0,\bullet}(X)$}.\]

\begin{definition}\label{d-gue50508d}
Let $D\subset U$ be an open set. We say that a function $u\in C^\infty(D)$ is rigid if $Tu=0$. We say that a function $u\in C^\infty(X)$ is Cauchy-Riemann (CR for short)
if $\ddbar_bu=0$. We call $u$ a rigid CR function if  $\ddbar_bu=0$ and $Tu=0$.
\end{definition}

\begin{definition} \label{d-gue150508dI}
Let $F$ be a complex vector bundle over $X$. We say that $F$ is rigid (CR) if 
$X$ can be covered with open sets $U_j$ with trivializing frames $\set{f^1_j,f^2_j,\dots,f^r_j}$, $j=1,2,\ldots$, such that the corresponding transition matrices are rigid (CR). The frames $\set{f^1_j,f^2_j,\dots,f^r_j}$, $j=1,2,\ldots$, are called rigid (CR) frames. 
\end{definition}

\begin{definition}\label{d-gue150514f}
Let $F$ be a complex rigid vector bundle over $X$ and let $\langle\,\cdot\,|\,\cdot\,\rangle_F$ be a Hermitian metric on $F$. We say that $\langle\,\cdot\,|\,\cdot\,\rangle_F$ is a rigid Hermitian metric if, for every rigid local frames $f_1,\ldots, f_r$ of $F$, we have $T\langle\,f_j\,|\,f_k\,\rangle_F=0$, for every $j,k=1,2,\ldots,r$. 
\end{definition} 

It is known that there is a rigid Hermitian metric on any rigid vector bundle $F$ (see Theorem 2.10 in~\cite{CHT} and Theorem 10.5 in~\cite{Hsiao14}). Note that  Baouendi-Rothschild-Treves~\cite{BRT85} proved that $T^{1,0}X$ is a rigid complex vector bundle over $X$.

From now on, let $E$ be a rigid CR vector bundle over $X$ and we take a rigid Hermitian metric $\langle\,\cdot\,|\,\cdot\,\rangle_E$ on $E$ and take a rigid Hermitian metric $\langle\,\cdot\,|\,\cdot\,\rangle$ on $\Complex TX$ such that 
$$T^{1,0}X\perp T^{0,1}X, \quad T\perp (T^{1,0}X\oplus T^{0,1}X), \quad 
\langle\,T\,|\,T\,\rangle=1.
$$
The Hermitian metrics on $E$ and $\Complex TX$ induce Hermitian metrics $\langle\,\cdot\,|\,\cdot\,\rangle$ and $\langle\,\cdot\,|\,\cdot\,\rangle_E$ on $T^{*0,\bullet}X$ and $T^{*0,\bullet}X\otimes E$, respectively. Let 
$$
A(x,y)\in (T^{*,\bullet}_yX\otimes E_y)^*\boxtimes (T^{*,\bullet}_xX\otimes E_x).
$$
 We write $\abs{A(x,y)}$ to denote the natural matrix norm of $A(x,y)$ induced by $\langle\,\cdot\,|\,\cdot\,\rangle_{E}$.
We denote by $dv_X=dv_X(x)$ the volume form on $X$ induced by the fixed 
Hermitian metric $\langle\,\cdot\,|\,\cdot\,\rangle$ on $\Complex TX$. Then we get natural global $L^2$ inner products $(\,\cdot\,|\,\cdot\,)_{E}$ and $(\,\cdot\,|\,\cdot\,)$
on $\Omega^{0,\bullet}(X,E)$ and $\Omega^{0,\bullet}(X)$, respectively. We denote by $L^2(X,T^{*0,q}X
\otimes E)$ and $L^2(X,T^{*0,q}X)$ the completions of $\Omega^{0,q}(X,E)$ and $\Omega^{0,q}(X)$ with respect to $(\,\cdot\,|\,\cdot\,)_{E}$ and $(\,\cdot\,|\,\cdot\,)$, respectively. Similarly, we denote by $L^2(X,T^{*0,\bullet}X
\otimes E)$ and $L^2(X,T^{*0,\bullet}X)$ the completions of $\Omega^{0,\bullet}(X,E)$ and $\Omega^{0,\bullet}(X)$ with respect to $(\,\cdot\,|\,\cdot\,)_{E}$ and $(\,\cdot\,|\,\cdot\,)$, respectively. We extend $(\,\cdot\,|\,\cdot\,)_{E}$ and $(\,\cdot\,|\,\cdot\,)$ to $L^2(X,T^{*0,\bullet}X\otimes E)$ and $L^2(X,T^{*0,\bullet}X)$ in the standard way, respectively.  For $f\in L^{2}(X,T^{*0,\bullet}X\otimes E)$, we denote $\norm{f}^2_{E}:=(\,f\,|\,f\,)_{E}$. Similarly, for $f\in L^{2}(X,T^{*0,\bullet}X)$, we denote $\norm{f}^2:=(\,f\,|\,f\,)$. 

We write $\ddbar_{b}$ to denote the tangential Cauchy-Riemann operator acting on forms with values in $E$:
\[\ddbar_{b}:\Omega^{0,\bullet}(X, E)\To\Omega^{0,\bullet}(X,E).\]
Since $E$ is rigid, we can also define $Tu$ for every $u\in\Omega^{0,q}(X,E)$ and we have 
\begin{equation}\label{e-gue150508d}
T\ddbar_{b}=\ddbar_{b}T\ \ \mbox{on $\Omega^{0,\bullet}(X,E)$}. 
\end{equation}
For every $m\in\mathbb Z$, let
\begin{equation}\label{e-gue150508dI}
\begin{split}
&\Omega^{0,q}_m(X,E) \, := \, \set{u\in\Omega^{0,q}(X,E);\, Tu=imu},\ \ q=0,1,2,\ldots,n, \nonumber \\
&\Omega^{0,\bullet}_m(X,E) \, := \, \set{u\in\Omega^{0,\bullet}(X,E);\, Tu=imu}. \nonumber
\end{split}
\end{equation} 
For each $m\in\mathbb Z$, we denote by $L^2_m(X,T^{*0,q}X\otimes E)$ and $L^2_m(X,T^{*0,q}X)$ the completions of $\Omega^{0,q}_m(X,E)$ and $\Omega^{0,q}_m(X)$ with respect to $(\,\cdot\,|\,\cdot\,)_{E}$ and $(\,\cdot\,|\,\cdot\,)$, respectively. Similarly, we denote by $L^2_m(X,T^{*0,\bullet}X\otimes E)$ and $L^2_m(X,T^{*0,\bullet}X)$ the completions of $\Omega^{0,\bullet}_m(X,E)$ and $\Omega^{0,\bullet}_m(X)$ with respect to $(\,\cdot\,|\,\cdot\,)_{E}$ and $(\,\cdot\,|\,\cdot\,)$, respectively. 


\subsection{Heat kernels of the Kohn Laplacians}\label{S:hkokl}

Since $T\ddbar_{b}=\ddbar_{b}T$, we have 
\[\ddbar_{b,m}:=\ddbar_{b}:\Omega^{0,\bullet}_m(X,E)\To\Omega^{0,\bullet}_m(X,E),\ \ \forall m\in\mathbb Z.\] 
We also write
\[\ol{\pr}^{*}_{b}:\Omega^{0,\bullet}(X,E)\To\Omega^{0,\bullet}(X,E)\]
to denote the formal adjoint of $\ddbar_{b}$ with respect to $(\,\cdot\,|\,\cdot\,)_E$.

Since $\langle\,\cdot\,|\,\cdot\,\rangle_E$ and $\langle\,\cdot\,|\,\cdot\,\rangle$ are rigid, we can check that 
\begin{equation}\label{e-gue150517i}
\begin{split}
&T\ddbar^{*}_{b}=\ddbar^{*}_{b}T\ \ \mbox{on $\Omega^{0,\bullet}(X,E)$}, \nonumber \\
&\ddbar^{*}_{b,m}:=\ddbar^{*}_{b}:\Omega^{0,\bullet}_m(X,E)\To\Omega^{0,\bullet}_m(X,E),\ \ \forall m\in\mathbb Z. \nonumber
\end{split}
\end{equation}
Now, we fix $m\in\mathbb Z$. The $m$-th Fourier component of Kohn Laplacian is given by
\begin{equation}\label{e-gue151113y}
\Box_{b,m}:=(\ddbar_{b,m}+\ddbar^{*}_{b,m})^2:\Omega^{0,\bullet}_m(X,E)\To\Omega^{0,\bullet}_m(X,E). \nonumber
\end{equation}
We extend $\Box_{b,m}$ to $L^{2}_m(X,T^{*0,\bullet}X\otimes E)$ by 
\begin{equation}\label{e-gue151113yI}
\Box_{b,m}:{\rm Dom\,}\Box_{b,m}\subset L^{2}_m(X,T^{*0,\bullet}X\otimes E)\To L^{2}_m(X,T^{*0,\bullet}X\otimes E)\,, \nonumber
\end{equation}
where 
$$
{\rm Dom\,}\Box_{b,m}:=\{u\in L^{2}_m(X,T^{*0,\bullet}X\otimes E);\, \Box_{b,m}u\in L^{2}_m(X,T^{*0,\bullet}X\otimes E)\}
$$ 
for which, for any $u\in L^{2}_m(X,T^{*0,\bullet}X\otimes E)$, $\Box_{b,m}u$ is defined in the sense of distribution. 
It is known that $\Box_{b,m}$ is self-adjoint, ${\rm Spec\,}\Box_{b,m}$ is a discrete subset of $[0,\infty[$ and, for every $\nu\in{\rm Spec\,}\Box_{b,m}$, $\nu$ is an eigenvalue of $\Box_{b,m}$ (see Section 3 in~\cite{CHT}).
For every $\nu\in{\rm Spec\,}\Box_{b,m}$, let $\set{f^\nu_1,\ldots,f^\nu_{d_{\nu}}}$ be an orthonormal frame for the eigenspace of $\Box_{b,m}$ with eigenvalue $\nu$. The heat kernel $e^{-t\Box_{b,m}}(x,y)$ is given by 
\begin{equation}\label{e-gue151023a}
e^{-t\Box_{b,m}}(x,y)=\sum_{\nu\in{\rm Spec\,}\Box_{b,m}}\sum^{d_{\nu}}_{j=1}e^{-\nu t}f^\nu_j(x)\otimes(f^\nu_j(y))^\dagger, \nonumber
\end{equation}
where $f^\nu_j(x)\otimes(f^\nu_j(y))^\dagger$ denotes the linear map: 
\[\begin{split}
f^\nu_j(x)\otimes(f^\nu_j(y))^\dagger:T^{*0,\bullet}_yX\otimes E_y&\To T^{*0,\bullet}_xX\otimes E_x,\\
u(y)\in T^{*0,\bullet}_yX\otimes E_y&\To f^\nu_j(x)\langle\,u(y)\,|\,f^\nu_j(y)\,\rangle_E\in T^{*0,\bullet}_xX\otimes E_x.\end{split}\]
Let 
$$
e^{-t\Box_{b,m}}:L^2(X,T^{*0,\bullet}X\otimes E)\To L^2_m(X,T^{*0,\bullet}X\otimes E)
$$ 
be the continuous operator with distribution kernel $e^{-t\Box_{b,m}}(x,y)$.


\subsection{Mellin transformation}\label{S:mtransf} 

Let $\Gamma(z)$ be the Gamma function on $\Complex$. Then, for ${\rm Re\,}z>0$, we have 
\[\Gamma(z)=\int^\infty_0e^{-t}t^{z-1}dt.\]
$\Gamma(z)^{-1}$ is an entire function on $\Complex$ and 
\begin{equation}\label{e-gue160313b}
\Gamma(z)^{-1}=z+O(z^2)\ \ \mbox{near $z=0$}. \nonumber
\end{equation}

We suppose that $f(t)\in C^\infty(\Real_+)$ verifies the following two conditions:
\begin{itemize}
\item[I.] 
\begin{equation}\label{e-gue160420g}
\mbox{$f(t)\sim \sum^{\infty}_{j=0}f_{-k+\frac{j}{2}}t^{-k+\frac{j}{2}}$ as $t\To0^+$},  \nonumber
\end{equation}
where $k\in\mathbb N_0$, $f_{-k+\frac{j}{2}}\in\mathbb C$, $j=0,1,2,\ldots$.
\item[II.]  For every $\delta>0$, there exist $c>0$, $C>0$ such that 
\begin{equation}\label{e-gue160420I}
\abs{f(t)}\leq Ce^{-ct},\ \ \forall t\geq\delta. \nonumber
\end{equation}
\end{itemize}

\begin{definition}\label{d-gue160313}
The \emph{Mellin transformation} of $f$ is the function defined by, for ${\rm Re\,}z>k$, 
\begin{equation}\label{e-gue160313bIII}
M[f](z)=\frac{1}{\Gamma(z)}\int^\infty_0 f(t)t^{z-1}dt.  \nonumber
\end{equation}
\end{definition}

We can repeat the proof of Lemma 5.5.2 in \cite{MM} and deduce the following, see \cite[Theorem 4.2]{HH} for the proof, 

\begin{theorem}\label{t-gue160313}
$M[f]$ extends to a meromorphic function on $\Complex$ with poles contained in 
\[\set{\ell-\frac{j}{2};\, \ell,j\in\mathbb Z},\] 
and its possible poles are simple. Moreover, $M[f]$ is holomorphic at $0$.
\end{theorem}


\subsection{Definition of the Quillen metric}\label{s-gue160502q}

In this subsection we recall the construction of the Fourier components of the analytic torsion for the rigid CR vector bundle $E$ over the CR manifold $X$ with a transversal CR $S^1$-action from \cite[\S 4]{HH}.

Let $N$ be the number operator on $T^{*0,\bullet}X$, i.e. $N$ acts on $T^{*0,q}X$ by multiplication by $q$.
Fix $q=0, 1, \cdots, n$, and take a point $x \in X$. Let $e_1(x), \cdots, e_d(x)$ be an orthonormal frame of $T_x^{*0,q}X \otimes E_x$. Let 
$$
A \in (T_x^{*0,\bullet}X \otimes E_x)^*\boxtimes (T_x^{*0,\bullet}X \otimes E_x).
$$ 
Put 
$$
\operatorname{Tr}^{(q)} A := \sum_{j=1}^d \langle Ae_j | e_j \rangle_E
$$ 
and set
\begin{eqnarray}
\operatorname{Tr}A:= \sum_{j=0}^n \operatorname{Tr}^{(j)}A, \nonumber \\
\operatorname{STr}A:= \sum_{j=0}^n (-1)^j \operatorname{Tr}^{(j)}A.
\end{eqnarray}
Let 
$$
A:C^\infty(X,T^{*0,\bullet}X\otimes E)\To C^\infty(X,T^{*0,\bullet}X\otimes E)
$$ 
be a continuous operator with distribution kernel 
$$
A(x,y)\in C^\infty(X\times X,(T^{*0,\bullet}_yX\otimes E_y)^*\boxtimes(T^{*0,\bullet}_xX\otimes E_x)).
$$ 
We set 
\[\operatorname{Tr}^{(q)}\lbrack A\rbrack:=\int_X\operatorname{Tr}^{(q)} A(x,x)dv_X(x)\] and put 
\begin{eqnarray}
\operatorname{Tr}\lbrack A\rbrack:= \sum_{j=0}^n \operatorname{Tr}^{(j)}[A], \nonumber \\
\operatorname{STr}\lbrack A\rbrack:= \sum_{j=0}^n (-1)^j \operatorname{Tr}^{(j)}[A].
\end{eqnarray}
Let 
$$
\Pi_m : L^2_m(X, T^{*0,\bullet}X \otimes E) \to \Ker \Box_{b,m}
$$  
be the orthogonal projection and let 
$$
\Pi^\perp_m : L^2_m(X, T^{*0,\bullet}X \otimes E) \to (\Ker\Box_{b,m})^\perp
$$  
be the orthogonal projection, where 
\[(\Ker\Box_{b,m})^\perp=\set{u\in L^2_m(X, T^{*0,\bullet}X \otimes E);\, (\,u\,|\,v\,)_E=0,\  \ \forall v\in\Ker \Box^E_{b,m}}.\]

By \cite[Theorem 1.7]{CHT}, we have the following asymptotic expansion: 
\begin{equation}\label{E:5.5.10b}
\operatorname{STr}  \lbrack N e^{-t \Box_{b,m}} \rbrack:=\int \operatorname{STr} (Ne^{-t\Box_{b,m}} )(x,x)dv_X(x)\sim\sum^\infty_{j=0} \hat{B}_{m,-n+\frac{j}{2}}t^{-n+\frac{j}{2}}\ \ \mbox{as $t\To0^+$}, 
\end{equation}
where $\hat{B}_{m,-n+\frac{j}{2}}\in\Complex$ independent of $t$, for each $j$. By using \eqref{E:5.5.10b} and Theorem \ref{t-gue160313}, cf. also \cite[\S 4]{HH}, we can show that, for $\operatorname{Re}(z)>n$, the following $\zeta$-function is well defined, 
\begin{equation}\label{E:5.5.12}
\theta_{b,m}(z) \, = \, - \mathcal{M} \left\lbrack \operatorname{STr}  \lbrack N e^{-t \Box_{b,m}} \Pi^\perp_m  \rbrack \right\rbrack   \, = \, - \operatorname{STr} \left\lbrack N ({\Box}_{b,m})^{-z} {\Pi}^\perp_m \right\rbrack,  
\end{equation}
where $\mathcal{M}$ denotes the Mellin transformation, cf. Definition~\ref{d-gue160313}.
Moreover, we can show that, cf. \cite[Lemma  4.4]{HH},
\begin{lemma}\label{L:meroext}
 $\theta_{b,m}(z)$ extends to a meromorphic function on $\mathbb{C}$ with poles contained in the set  
\[\set{\ell-\frac{j}{2};\, \ell,j\in\mathbb Z},\]
its possible poles are simple, and $\theta_{b,m}(z)$ is holomorphic at $0$.
\end{lemma}

We can now introduce the definition of the $m$-th Fourier component of the analytic torsion for $X$ with $S^1$ action, cf. \cite[Definition 4.5]{HH}.

\begin{definition}\label{d-gue160502w}
Fix $m\in\mathbb Z$. We define $\exp ( -\frac{1}{2} \theta_{b,m}'(0) )$ the $m$-th Fourier component of the analytic torsion for the rigid vector bundle $E$ over the CR manifold $X$ with transversal CR $S^1$-action. 
\end{definition}

Put
$$
\ddbar_{b,m}\, := \, \ddbar_b:\Omega^{0,q}(X)\rightarrow\Omega^{0,q+1}(X)
$$ 
with a $\ddbar_{b,m}$-complex: 
\[
\ddbar_{b,m} : \cdots \rightarrow\Omega^{0,q-1}_m(X) \rightarrow \Omega^{0,q}_m(X)\rightarrow\Omega^{0,q+1}_m(X) \rightarrow \cdots.
\]

\begin{definition}\label{D:krcohgp}
For each $m \in \mathbb{Z}$ and $q=0,1, \cdots, n$, the cohomology group:
\[
H^q_{b,m}(X) \, := \, \frac{\operatorname{Ker} \overline{\partial}_{b,m}:\Omega^{0,q}_m(X) \to \Omega^{0,q+1}_m(X)}{\operatorname{Im}\overline{\partial}_{b,m}:\Omega^{0,q-1}_m(X) \to \Omega^{0,q}_m(X)}
\]
is called the $m$-th $S^1$ Fourier component of the $q$-th $\overline{\partial}_{b}$ Kohn-Rossi cohomology group.
\end{definition}
By Theorem 3.7 of \cite{CHT}, for each $m \in \mathbb{Z}$ and $q=0,1, \cdots, n$, the cohomology group $H^q_{b,m}(X,E)$ is isomorphic to the kernel of $\Box^{(q)}_{b,m}$ and $\dim H^q_{b,m}(X,E)< \infty$. 
Denote by 
$$
H^\bullet_{b, m}(X, E) = \oplus_{q=0}^n H^q_{b,m}(X,E).
$$ 
For a finite dimensional vector space $V$, we set 
$$
\det V := \wedge^{\text{max}}V.
$$
We then denote by 
$$
(\det V)^{-1}:= (\det V)^*,
$$ the dual line of $\det V$.
Then
\[
\det H^\bullet_{b,m}(X,E) = \otimes_{q=0}^n \left( \det H^q_{b,m}(X,E)  \right)^{(-1)^q}
\]
is the determinant line of the cohomology $H^\bullet_{b,m}(X,E)$. We define 
\begin{equation}\label{E:5.5.14}
\lambda_{b,m}(E) = \left( \det H^\bullet_{b,m}(X,E)  \right)^{-1}. \nonumber
\end{equation}
The rigid Hermitian metrics $\langle \, \cdot \, |\, \cdot \, \rangle$ and $\langle \, \cdot \, |\, \cdot \, \rangle_E$ on $\mathbb{C}TX$ and $E$, respectively, induce a canonical $L^2$-metric $h^{H^\bullet_{b,m}(X,E)}$ on $H^\bullet_{b,m}(X,E)$. Let $|\cdot|_{\lambda_{b,m}(E)}$ be the $L^2$-metric on $\lambda_{b,m}(E)$ induced by $h^{H^\bullet_{b,m}(X,E)}$. 

Now we can define the Quillen metric on $\det H^\bullet_{b,m}(X,E)$.

\begin{definition}\label{D:5.5.5}
Fix $m \in \mathbb{Z}$. The Quillen metric $\| \cdot  \|_{\lambda_{b, m}(E)}$ on $\det H^\bullet_{b,m}(X,E)$ is defined as
\[
\| \cdot  \|_{\lambda_{b, m}(E)} \, := \,  |\cdot|_{\lambda_{b,m}(E)} \cdot \exp ( -\frac{1}{2} \theta_{b,m,E}'(0)  ).
\]
\end{definition}


\subsection{Tangential de Rham cohomology group and tangential characteristic classes}\label{S:ts}
For every $r \, = \, 0, 1, 2, \cdots, 2n$, put 
$$
\Omega^r_0(X) \, = \, \left\{  u \in \oplus_{p+q=r} \Omega^{p,q}(X); \, Tu=0 \, \right\}
$$
and set 
$$
\Omega^\bullet_0(X) \, =\, \oplus_{r=0}^{2n}\Omega_0^r(X).
$$
 Since $Td = dT$ (see \eqref{e-gue150508d}), we have $d$-complex:
\[
d:\cdot \to \Omega^{r-1}_0(X) \to \Omega^r_0(X) \to \Omega^{r+1}_0(X) \to \cdots
\]
and we define the $r$-th tangential de Rham cohomology group: 
$$
\; \mathcal{H}^r_{b,0}(X) \, := \, \frac{\operatorname{Ker} d:\Omega^r_0(X) \to \Omega^{r+1}_0(X)}{\operatorname{Im} d:\Omega^{r-1}_0(X) \to \Omega^{r}_0(X)}.
$$ 
Put
 $$
 \mathcal{H}^\bullet_{b,0}(X) \, = \, \oplus^{2n}_{r=0}\mathcal{H}^r_{b,0}(X).
 $$
Let $F$ be a rigid complex vector bundle over $X$ of rank $r$. It was shown in \cite[Theorem 2.11]{CHT} that there is a rigid connection $\nabla$ on $F$, that is, for any rigid local frame $f \, = \, (f_1,f_2,\cdots,f_r)$ of $F$ on an open set $D \subset X$, the connection matrix $\theta(\nabla, f) \, = \, (\theta_{j,k})^r_{j,k=1}$ satisfies $\theta_{j,k} \in \Omega^1_0(D)$, for every $j, k = 1, \cdots, r$. Let 
$$
\Theta(\nabla, F) \in C^\infty(X, \wedge^2(\mathbb{C}T^*X)\otimes \operatorname{End}(F))
$$ 
be the associated curvature. Let 
$$
h(z) = \sum_{j=0}^\infty a_jz^j, a_j \in \mathbb{R},
$$ 
for every $j$, be a real power series on $z \in \mathbb{C}$. Set
\[
H(\Theta(\nabla, F)) \, = \, \operatorname{Tr} \left( h(\frac{i}{2\pi}\Theta(\nabla, F))\right).
\]
It is clear that 
$$
H(\Theta(\nabla,F)) \, \in \, \Omega^\bullet_0(X)
$$ 
and is known that $H(\Theta(\nabla,F))$ is a closed differential form and the tangential de Rham cohomology class
$$
\left[H(\Theta(\nabla,F))\right] \in \mathcal{H}^\bullet_{b,0}(X)
$$ 
does not depend on the choice of rigid connection $\nabla$, cf. \cite[Theorem 2.5, Theorem 2.6]{CHT}. Put 
\begin{equation}
\operatorname{ch}_b(\nabla, F) \, =  \, \operatorname{ch}_b(\Theta(\nabla, F)) := \, H(\Theta(\nabla, F)) \, \in \, \Omega^\bullet_0(X), \nonumber
\end{equation}
where $h(z)\, = \, e^z$ and set 
\begin{equation}
\operatorname{Td}_b(\nabla, F) \, = \, \operatorname{Td}_b(\Theta(\nabla, F))  \, := \, e^{H(\Theta(\nabla,F))} \, \in \, \Omega^\bullet_0(X), \nonumber
\end{equation}
where $h(z) \, = \, \log (\frac{z}{1-e^{-z}})$. We now introduce tangential Todd class and tangential Chern character.
\begin{definition}
Tangential Chern character of $F$ is given by 
\[
\operatorname{ch}_b(F) \, := \, [\operatorname{ch}_b(\nabla, F)] \, \in \, \mathcal{H}^\bullet_{b,0}(X),
\]
and tangential Todd class of $F$ is given by
\[
\operatorname{Td}_b(F) \, := \, [\operatorname{Td}_b(\nabla, F)] \, \in \, \mathcal{H}^\bullet_{b,0}(X).
\] 
\end{definition}

Let 
$$
P  \, = \, \left\{\, u \in \oplus^n_{p=0}\Omega^{p,p}(X); \, Tu = 0  \, \right\}.
$$ 
Let $P' \subset P$ be the set of smooth forms $\alpha \in P$ such that there exist smooth forms $\beta, \gamma \in \Omega^\bullet_0(X)$ for which 
$$
\alpha \, = \, \partial_b \beta + \overline{\partial}_b \gamma.
$$ 
When $\alpha, \alpha' \in P$, we write $\alpha \equiv \alpha'$ if $\alpha - \alpha' \in P'.$ We can check that if $\eta \in P$ is closed and has compact support and $\alpha \equiv \alpha'$, then
\[
\int_X \alpha \wedge \eta \wedge \omega_0 \, = \, \int_X \alpha' \wedge \eta \wedge \omega_0.
\]
Hence the pairing of elements of $P/P'$ with such $\eta$ is well-defined. Let $\nabla'$ be a rigid connection induced by another rigid Hermitian metric $\langle\,\cdot\,|\,\cdot\,\rangle'_F$ on $F$. By \cite[\S (f)]{BGS1}, we have the unique secondary tangential characteristic (Bott-Chern) classes $\widetilde{\operatorname{Td}}_b(\nabla, \nabla', F)$ and $\widetilde{\operatorname{ch}}_b( \nabla, \nabla', F)$ in $P/P'$ such that
\begin{eqnarray}
\frac{\overline{\partial}_b\partial_b}{2\pi\sqrt{-1}} \widetilde{\operatorname{Td}}_b(\nabla, \nabla', F) \, = \, \operatorname{Td}_b(\nabla', F) - \operatorname{Td}_b(\nabla, F), \nonumber \\
\frac{\overline{\partial}_b\partial_b}{2\pi\sqrt{-1}} \widetilde{\operatorname{ch}}_b(\nabla, \nabla', F) \, = \, \operatorname{ch}_b(\nabla', F) - \operatorname{ch}_b(\nabla, F). \nonumber
\end{eqnarray}

Baouendi-Rothschild-Treves \cite{BRT85} proved that $T^{1,0}X$ is a rigid complex vector bundle over $X$. Thus, we can define tangential Todd class of $T^{1,0}X$, tangential Chern character of $T^{1,0}X$ and tangential Bott-Chern classes of $T^{1,0}X$.

\subsection{Main Theorem}\label{S:mainth}
In this subsection we state the main result of this paper. Let $E$ be a rigid complex vector bundle over a compact connected strongly pseudoconvex CR manifold $X$ of dimension $2n+1, n \ge 1$ with a transversal CR $S^1$ action on $X$. Let $\nabla^{TX}$ and $\nabla'^{TX}$ be the Levit-Civita connections on $TX$ with respect to the metrics $\langle\, \cdot \, | \, \cdot \, \rangle$ and $\langle \, \cdot \, | \, \cdot \, \rangle'$ on $\mathbb{C}TX$, respectively. Let $P_{T^{1,0}X}$ be the natural projection from $\mathbb{C}TX$ onto $T^{1,0}X$. Then, 
$$
\nabla^{T^{1,0}X}\, := \, P_{T^{1,0}X}\nabla^{TX} \quad  \text{and}  \quad \nabla'^{T^{1,0}X}\, := \, P_{T^{1,0}X}\nabla'^{TX}
$$ 
are connections on $T^{1,0}X$. Let $\nabla^E$ and $\nabla'^E$ be the connections on $E$ induced by the rigid Hermitian metrics $\langle\,\cdot\,|\,\cdot\,\rangle_E$ and $\langle\, \cdot\, | \, \cdot \, \rangle'_E$ on $E$, respectively. We can check that $\nabla^{T^{1,0}X}, \nabla'^{T^{1,0}X}, \nabla^E$ and $\nabla'^E$ are rigid. We denote by $|| \cdot ||'_{\lambda_{b,m}(E)}$ the Quillen metric induced by the metrics $\langle\, \cdot \, | \, \cdot \, \rangle'$ and $\langle\, \cdot \, | \, \cdot \, \rangle'_E$. Denote by $\widetilde{\operatorname{Td}}_{b}(\nabla^{T^{1,0}X}, \nabla'^{T^{1,0}X}, T^{1,0}X)$ the secondary tangential Todd class for the vector bundle $T^{1,0}X$,
$\widetilde{\operatorname{ch}}_{b}(\nabla^{E}, \nabla'^{E}, E)$ the secondary tangential Chern character for the vector bundle $E$, $\operatorname{Td}_{b}(\nabla'^{T^{1,0}X}, T^{1,0}X)$ the Todd class for the vector bundle $T^{1,0}X$ and $\operatorname{ch}_b(\nabla^E, E)$ the Chern character for the vector bundle $E$.

The following theorem is the main result of this paper.

\begin{theorem}\label{T:main}
The following identity holds:
\begin{eqnarray}
\log \left( \,  \frac{|| \cdot ||'_{\lambda_{b,m}(E)} }{|| \cdot ||_{\lambda_{b,m}(E)} }   \, \right)  \, = \, \frac{1}{2\pi}  \int_X \widetilde{\operatorname{Td}}_{b}(\nabla^{T^{1,0}X}, \nabla'^{T^{1,0}X}, T^{1,0}X) \wedge \operatorname{ch}_b(\nabla^E, E) \wedge e^{-m\frac{d\omega_0}{2\pi}} \wedge \omega_0 \nonumber \\
 + \frac{1}{2\pi} \int_X \operatorname{Td}_{b}(\nabla'^{T^{1,0}X}, T^{1,0}X) \wedge \widetilde{\operatorname{ch}}_{b}(\nabla^{E}, \nabla'^{E}, E) \wedge e^{-m \frac{d\omega_0}{2\pi}} \wedge \omega_0. \nonumber
\end{eqnarray}
\end{theorem}


\subsection{Anomaly formula for some class of orbifold line bundles}\label{S:orbi}
In \cite{M}, Ma first introduced analytic torsion on orbifolds and anomaly formula for Quillen metrics in the case of orbifolds, which is expressed explicitly in the form of characteristic and Bott-chern characteristic classes. Comparing with Ma's formula, we get a simpler anomaly formula for some class of orbifold line bundles from our main result, Theorem \ref{T:main}. We first recall some backgrounds on orbifold geometry. We will follow the presentation of \cite[Subsection 1.4]{CHT} closely. 

Let $M$ be a manifold and let $G$ be a compact Lie group. Assume that $M$ admits a $G$-action:
\begin{eqnarray}
G \times M \to M, \nonumber \\
(g, x) \to g \circ x. \nonumber
\end{eqnarray}
We assume that the action of $G$ on $M$ is locally free, that is, for every point $x \in M$, the stabilizer group
\[
G_x \, :=\, \left\{ g \in G; g \circ x = x \right\}
\]
of $x$ is a finite subgroup of $G$. It is well known that, in such a case, the quotient space
\begin{equation}\label{E:m/g}
X \, := \, M/G
\end{equation}
is an orbifold. A theorem of Satake \cite{Sa} says that the converse is also true: every orbifold has a presentation of the form \eqref{E:m/g}. We assume that $M$ is a compact connected complex manifold with complex structure $T^{1,0}M$. Then the group $G$ induces an action on $\mathbb{C}TM$:
\[
\begin{split}
G \times \mathbb{C}TM \to \mathbb{C}TM, 
(g, u) \to g^*u,
\end{split}
\]
where $g^*=(g^{-1})_*$ denotes the push-forward by $g^{-1}$ on $\mathbb{C}TM$. We assume that $G$ acts holomorphically, that is, 
$$
g^*(T^{1,0}M) \subset T^{1,0}M,
$$ 
for every $g \in G$.
Let $T^{0,1}M := \overline{T^{1,0}M}$. Put
\[
\mathbb{C}T(M/G) \, := \, \mathbb{C}TM/G, \quad T^{1,0}(M/G) \, := \, T^{1,0}M/G, \quad  T^{0,1}(M/G) \, := \, T^{0,1}M/G.
\]
Assume that 
$$
T^{1,0}(M/G) \cap T^{0,1}(M/G) \, = \, \left\{ 0 \right\}.
$$
Then, $T^{1,0}(M/G)$ is a complex structure on $M/G$ and $M/G$ is a complex obifold. Suppose that 
\[
\dim_{\mathbb{C}}T^{1,0}(M/G) \, = \, n.
\]
Let $L$ be a $G$-invariant holomorphic line bundle over $M$, that is, for every transition function $h$ of $L$ on an open set $U \subset M$, we have $h(g \circ x) = h(x)$, for every $g \in G, \, x \in G$ with $g \circ x \in U.$ Suppose that $L$ admits a locally free-$G$ action:
\[
\begin{split}
G \times L \to L, \nonumber \\
(g, x) \to g \circ x, \nonumber
\end{split}
\]
where 
$$
\pi(g \circ x) \, =\, g \circ (\pi(x)),
$$
for every $g \in G$, where $\pi\, : \, L \to M$ denotes the natural projection, and where the action of $G$ on $L$ is linear on the fibers of $L$, that is, for every $g \in G$, every $z \in M$, we have 
\[
g \circ \left( s(z) \otimes \lambda \right) \,  = \, s_1(g \circ z) \otimes \rho(g, z)\lambda, 
\]
for every $\lambda \in \mathbb{C}$, where $s$ and $s_1$ are local sections of $L$ defined near $z$ and $g \circ z$, respectively, and $\rho(g, z) \in \mathbb{C}$ depends on $z$ and $g$ smoothly. Then, $L/G$ is an orbifold holomorphic line bundle over $M/G$. For every $m \in \mathbb{N}$, let $L^m$ be the $m$-th tensor power of $L$. Then, the $G$-action on $L$ induces a locally free $G$-action on $L^m$:
\[
\begin{split}
G \times L^m \to L^m, \\
(g, x) \to g \circ x,
\end{split}
\]
where
\[
\pi_m(g \circ x) \, = \, g \circ (\pi_m(x)),
\]
for every $g \in G$, where $\pi_m \, : \, L^m \to M$ denotes the natural projection, and the action of $G$ on $L^m$ is linear on the fibers of $L^m$. Then, $L^m/G$ is again an orbifold holomorphic line bundle over $M/G$. Now, we fix $m \in \mathbb{Z}$. Let $T^{*0,q}M$ denote the bundle of $(0,q)$ forms on $M$. Since $G$ action is holomorphic, $G$ induces a natural action on $T^{*0,q}M \otimes L^m$:
\[
\begin{split} 
G \times (T^{*0,q}M \otimes L^m) \to T^{*0,q}M \otimes L^m, \\
(g, u) \to g^*u.
\end{split}
\]
For every $q \, = \, 0, 1, 2, \cdots, n$, put
\begin{equation}
\Omega^{0,q}(M/G, L^m/G) \, := \, \left\{ u \in \Omega^{0,q}(M, L^m); \, g^*u = u, \ \forall g \in G \right\}, \nonumber
\end{equation}
where $\Omega^{0,q}(M, L^m)$ denotes the space of smooth sections with values in $T^{*0,q}M \otimes L^m$. The Cauchy-Riemann operator 
\[
\overline{\partial} \, : \, \Omega^{0,q}(M/G, L^m/G) \to \Omega^{0,q+1}(M/G, L^m/G)
\]
is $G$-invariant and we have the following $\overline{\partial}$-complex:
\[
\overline{\partial}: \cdots \to \Omega^{0,q-1}(M/G, L^m/G) \to \Omega^{0,q}(M/G, L^m/G) \to \Omega^{0,q+1}(M/G, L^m/G) \to \cdots
\]
and, hence, we can consider the $q$-th Dolbeault cohomology group:
\[
H^q(M/G, L^m/G) \, := \, \frac{\Ker \overline{\partial} \, : \, \Omega^{0,q}(M/G, L^m/G) \to \Omega^{0,q+1}(M/G, L^m/G)}{\operatorname{Im} \overline{\partial} \, : \, \Omega^{0,q-1}(M/G, L^m/G) \to \Omega^{0,q}(M/G, L^m/G)}
\]

Let $L^*$ be the dual bundle of $L$. Then, $L^*$ is also a $G$-invariant holomorphic line bundle and $L^*$ admits a locally free $G$-action:
\[
\begin{split}
G \otimes L^* \to L^*, \\
(g, x) \to g \circ x
\end{split}
\] 
where 
\[
\pi^*(g \circ x) \, = \, g \circ (\pi^*(x)),
\]
for every $g \in G$, where $\pi^* \, : \, L^* \to M$ denotes a natural projection, and the action of $G$ on $L^*$ is linear on the fibers of $L^*$. Then, $L^*/G$ is also an orbifold holomorphic line bundle over $M/G$. Let $\operatorname{Tot}(L^*)$ be the space of all non-zero vectors of $L^*$. Assume that $\operatorname{Tot}(L^*)/G$ is a smooth manifold. Take any $G$-invariant Hermitian fiber metric $h^{L^*}$ on $L^*$, set
\[
\widetilde{X} \, = \, \left\{ v \in L^*; |v|_{h^{L^*}} \, = \, 1 \right\}
\]
and put 
$$
X \, = \, \widetilde{X}/G.
$$
Since $\operatorname{Tot}(L^*)$ is a smooth manifol, $X \, = \, \widetilde{X}/G$ is a smooth manifold. The natural $S^1$ action on $\widetilde{X}$ induces a locally free action $S^1$ action $e^{i\theta}$ on $X$. Moreover, we can check that $X$ is a CR manifold and the $S^1$ action on $X$ is CR and transversal. We will use the same notations as before. We can repeat the proof of \cite[Theorem 1.2]{CHT} with minor changes and show that, for every $q \, =\, 0, 1, 2, \cdots, n,$ and every $m \in \mathbb{Z}$, we have
\begin{equation}\label{E:orblineiso}
\begin{split}
H^q(M/G, L^m/G) \, \cong \, H^q_{b,m}(X), \\
\dim H^q(M/G, L^m/G) \, = \, \dim H^q_{b,m}(X).
\end{split}
\end{equation}

We pause and introduce some notations. For every $x \in \operatorname{Tot}(L^*)$ and $g \in G$, put 
\[ 
N(g, x)\, = \,
  \begin{cases}
    1,       & \quad \text{if } g \not\in G_x \\
    \operatorname{inf}\left\{ l \in \mathbb{N}; g^l \, = \, \operatorname{Id} \right\},  & \quad \text{if } g \in G_x.
  \end{cases}
\]
Set
\[
p \, = \, \inf \left\{ N(g, x); \, x \in  \operatorname{Tot}(L^*), \ g \in G, \ g \not= \operatorname{Id} \right\},
\]
where $\operatorname{Id}$ denotes the identity element of $G$. It is known that $X_p$ is open and dense subset of $X$. Recall that in this work we work with $p=1$. From Theorem \ref{T:main} and \eqref{E:orblineiso}, we deduce
\begin{theorem}
With the notations used above, recall that we work with the assumptions that $M$ is connected and $\operatorname{Tot}(L^*)$ is smooth. Fix a Hermitian metric on $L^m/G$. Then, for every $m \in \mathbb{Z}$, we have
\begin{eqnarray}
\log \left( \,  \frac{|| \cdot ||'_{\lambda(L^m/G)} }{|| \cdot ||_{\lambda(L^m/G)} }   \, \right)  \, = \, \frac{1}{2\pi}  \int_X \widetilde{\operatorname{Td}}_{b}(\nabla^{T^{1,0}X}, \nabla'^{T^{1,0}X}, T^{1,0}X)  \wedge e^{-m\frac{d\omega_0}{2\pi}} \wedge \omega_0. \nonumber \\
\end{eqnarray}
\end{theorem}


\section{Asymptotic expansion of heat kernels}
In this section we introduce the complex tangential $\ast$-operator and a certain asymptotic expansion of heat kernels. The main result is Theorem \ref{T:5.5.6} which can be viewed as a CR analogue of \cite[Theorem 1.18]{BGS1} (cf. also \cite[Theorem 5.5.6]{MM}).

\subsection{Asymptotic expansion for heat kernels of the Kohn Laplacians}

We now define the complex tangential Hodge $\ast$-operator, see also \cite[Proposition 8.8]{DT}, as a complex conjugate linear map
\[
\ast_b : \Omega^{p,q}(X) \to \Omega^{n-p,n-q}(X)
\]
such that
\[
\langle\, \phi\, | \, \psi \,\rangle \frac{(d\omega_0)^n}{n!} \, = \, \phi \wedge \ast_b \psi, \quad \ast_b \ast_b \phi = (-1)^{p+q} \phi,
\]
for any $\phi, \psi \in \Omega^{p,q}(X)$.

We denote by $H^*X$ the dual bundle of $HX$ and $conj$ the natural conjugate map induced by the bundle automorphism
\begin{equation}
H^*X\otimes_{\mathbb{R}}\mathbb{C} \to H^*X\otimes_{\mathbb{R}}\mathbb{C}, \quad u \otimes \lambda \mapsto u \otimes \overline{\lambda},
\end{equation}
for any $u \in H^*X, \lambda \in \mathbb{C}.$
Then 
$$
\widehat{\ast}_b := conj\ast_b
$$ 
is a complex linear map. Clearly, 
$$
\partial_b = conj \, \overline{\partial}_b \, conj \, : \, \Omega^{p,q}(X) \to \Omega^{p+1,q}(X)
$$
and
$$
\widehat{\ast}_b \, = \, conj \, \ast_b \, = \, \ast_b \, conj \, : \, \Omega^{p,q}(X) \to \Omega^{n-q,n-p}(X).
$$

Let $(\, \cdot \, |\, \cdot \, )$ be the $L^2$ inner product on $\Omega^{p,q}(X)$ induced by $\langle \, \cdot \, | \, \cdot \, \rangle$. Then, for all $\phi, \psi \in \Omega^{p,q}(X)$,
\[
(\, \phi \, |\, \psi\,) \, =\, \int_X \langle \, \phi \,| \, \psi \, \rangle dv_X \, =\, \int_X  \phi  \wedge  \ast_b  \psi   \wedge \omega_0,
\]
where
\begin{equation}\label{E:volume}
dv_X \, = \,  \frac{(d\omega_0)^n}{n!} \wedge \omega_0
\end{equation}
is the volume form.

We can easily check that
\[
\partial_b^* \phi = -\ast_b \partial_b \ast_b \phi = -\widehat{\ast}_b\, \overline{\partial}_b\, \widehat{\ast}_b\, \phi, 
\]
and
\[
 \overline{\partial}_b \psi = - \ast_b \overline{\partial}_b \ast_b \psi = - \widehat{\ast}_b\, {\partial}_b\, \widehat{\ast}_b\, \psi. 
\]

Denote by $\mu: E \to E^*$ the induced conjugate linear bundle isomorphism from the vector bundle $E$ to its dual vector bundle $E^*$. Let $(\, \cdot \, |\, \cdot \, )_E$ be the $L^2$ inner product on $\Omega^{p,q}(X,E)$ induced by $\langle \, \cdot \, | \, \cdot \, \rangle$ and $\langle \, \cdot \, | \, \cdot \, \rangle_E$. Then, for all $\alpha, \beta \in \Omega^{p,q}(X, E)$,
\[
(\, \alpha \, |\, \beta \,)_E \, =\, \int_X \langle \, \alpha \,| \, \beta \, \rangle_E dv_X \, =\, \int_X  \alpha  \wedge  \left(\ast_b \otimes \mu \right) \beta   \wedge \omega_0,
\]
where $dv_X$ is the volume form defined in $\eqref{E:volume}$. We write $\ddbar'_{b}$ to denote the tangential Cauchy-Riemann operator acting on forms with values in $E^*$:
\[
\ddbar'_{b}:\Omega^{0,\bullet}(X, E^*)\To\Omega^{0,\bullet}(X,E^*).
\]
We can check that the adjoint of $\overline{\partial}_{b}$ is 
\begin{equation}
\overline{\partial}^{*}_{b} = - \left( \ast_b^{-1} \otimes \mu \right)^{-1} \overline{\partial}'_{b} \left( \ast_b \otimes \mu \right). \nonumber
\end{equation}

Let $\langle \, \cdot \, |\, \cdot \, \rangle_{s}$ and $\langle \, \cdot \, |\, \cdot \, \rangle_{E, s}, s \in [0, 1]$ be smooth families of rigid Hermitian metrics on $TX$ and $E$, respectively, such that 
$$
\langle \, \cdot \, |\, \cdot \, \rangle_{0} \, := \, \langle \, \cdot \, |\, \cdot \, \rangle, \quad \langle \, \cdot \, |\, \cdot \, \rangle_{1} \, := \, \langle \, \cdot \, |\, \cdot \, \rangle' \quad \text{and} \quad \langle \, \cdot \, |\, \cdot \, \rangle_{E, 0} \, := \, \langle \, \cdot \, |\, \cdot \, \rangle_E, \quad  \langle \, \cdot \, |\, \cdot \, \rangle_{E, 1} \, := \, \langle \, \cdot \, |\, \cdot \, \rangle'_E.
$$ 
Let $(\, \cdot \, |\, \cdot \, )_{E,s}$ be the $L^2$ inner products on $\Omega^{p,q}(X,E)$ induced by $\langle \, \cdot \, | \, \cdot \, \rangle_s$ and $\langle \, \cdot \, | \, \cdot \, \rangle_{E,s}$. Let $\ast_{b,s}$ be the tangential Hodge $\ast$-operators associated to the metrics $\langle \, \cdot \, |\, \cdot \, \rangle_{s}$ and $\mu_s$ be the induced conjugate linear bundle isomorphisms of $E$ and $E^*$ associated to the metric $\langle \, \cdot \, |\, \cdot \, \rangle_{E, s}$. Let 
$$
\Box_{b,s}:= \overline{\partial}_{b}\overline{\partial}^{\ast}_{b,s} + \overline{\partial}^{\ast}_{b,s}\overline{\partial}_b,
$$ 
where $\overline{\partial}^{\ast}_{b,s}$ denote the formal adjoint of $\overline{\partial}_b$ with respect to the $L^2$ scalar product $( \, \cdot \, | \, \cdot \, )_{E,s}$. We denote by 
$$
\Box_{b,m, s}:=\Box_{b, s}|_{\Omega^{0,\bullet}_m(X,E)}.
$$ 
Let $\| \cdot  \|_{\lambda_{b, m}(E), s}$ be the corresponding Quillen metrics  on $\det H^\bullet_{b,m}(X,E)$. Set 
\begin{equation}\label{E:qbs}
Q_{b, s} = - \left(\ast_{b,s} \otimes \mu_s \right)^{-1} \frac{\partial \left(\ast_{b, s} \otimes \mu_s \right)}{\partial s} = - \left(\ast^{-1}_{b, s} \frac{\partial \ast_{b, s}}{\partial s} + (\mu_s)^{-1} \frac{\partial \mu_s}{\partial s} \right).
\end{equation}
The following theorem is an analogue of \cite[Theorem 1.18]{BGS1} (cf. also \cite[Theorem 5.5.6]{MM}). 
\begin{theorem}\label{T:5.5.6}
 As $t \to 0^+$, for any $k \in \mathbb{N}$, there is an asymptotic expansion
\begin{equation}\label{E:5.5.18}
\operatorname{STr} \left\lbrack Q_{b,s} \exp \left( -t\Box_{b,m, s} \right) \right\rbrack = \sum_{j=0}^{k+2n} M_{-n+\frac{j}{2},s} t^{-n+\frac{j}{2}} +O(t^{\frac{k+1}{2}}), \nonumber
\end{equation}
where
\begin{equation}\label{E:5.5.19}
M_{0,s} = \frac{\partial}{\partial s} \log \left(  \| \cdot  \|^{2}_{\lambda_{b, m}(E), s} \right).
\end{equation}
\end{theorem}
\begin{proof}
By the small time asymptotic expansion for the heat kernel of the Kohn Laplacians in \cite[Theorem 1.7]{CHT} and proceeding formally as in the proof of \cite[Theorem 1.18]{BGS1}, we get \eqref{E:5.5.19}.
\end{proof}


\section{Anomaly formula of Analytic torsion on CR manifolds with $S^1$-action}

In this section we study the dependence of the analytic torsion under a change of the metrics. In Subsection \ref{SS:brt}, we recall the BRT trivializations from \cite{BRT85}. In Subsection \ref{SS:lochk}, we review the local heat kernels on the BRT trivializations. In Subsection \ref{SS:lochkpara}, we discuss certain local heat kernels depending on some parameters. In Subsection \ref{SS:const}, we derive the constant term of heat kernel asymptotics of the modified Kohn Laplacians on BRT trivializations. Finally, in Subsection \ref{SS:proof}, we give the proof of  our main theorem.

\subsection{BRT trivializations}\label{SS:brt}

To prove Theorem~\ref{T:main}, we need some preparations. We first need the following result due to Baouendi-Rothschild-Treves~\cite{BRT85}.

\begin{theorem}\label{t-gue150514}
For every point $x_0\in X$, we can find local coordinates $x=(x_1,\cdots,x_{2n+1})=(z,\theta)=(z_1,\cdots,z_{n},\theta), z_j=x_{2j-1}+ix_{2j},j=1,\cdots,n, x_{2n+1}=\theta$, defined in some small neighborhood $D=\{(z, \theta): \abs{z}<\delta, -\varepsilon_0<\theta<\varepsilon_0\}$ of $x_0$, $\delta>0$, $0<\varepsilon_0<\pi$, such that $(z(x_0),\theta(x_0))=(0,0)$ and 
\begin{equation}\label{e-can}
\begin{split}
&T=\frac{\partial}{\partial\theta} \nonumber \\
&Z_j=\frac{\partial}{\partial z_j}+i\frac{\partial\varphi}{\partial z_j}(z)\frac{\partial}{\partial\theta},j=1,\cdots,n \nonumber
\end{split}
\end{equation}
where $Z_j(x), j=1,\cdots, n$, form a basis of $T_x^{1,0}X$, for each $x\in D$ and $\varphi(z)\in C^\infty(D,\mathbb R)$ independent of $\theta$. We call $(D,(z,\theta),\varphi)$ BRT trivialization.
\end{theorem}

By using BRT trivialization, we get another way to define $Tu, \forall u\in\Omega^{0,q}(X)$. Let $(D,(z,\theta),\varphi)$ be a BRT trivialization. It is clear that
$$\{d\overline{z_{j_1}}\wedge\cdots\wedge d\overline{z_{j_q}}, 1\leq j_1<\cdots<j_q\leq n\}$$
is a basis for $T^{\ast0,q}_xX$, for every $x\in D$. Let $u\in\Omega^{0,q}(X)$. On $D$, we write
\begin{equation}\label{e-gue150524fb}
u=\sum\limits_{1\leq j_1<\cdots<j_q\leq n}u_{j_1\cdots j_q}d\overline{z_{j_1}}\wedge\cdots\wedge d\overline{z_{j_q}}. \nonumber
\end{equation}
Then, on $D$, we can check that
\begin{equation}\label{lI}
Tu=\sum\limits_{1\leq j_1<\cdots<j_q\leq n}(Tu_{j_1\cdots j_q})d\overline{z_{j_1}}\wedge\cdots\wedge d\overline{z_{j_q}} \nonumber
\end{equation}
and $Tu$ is independent of the choice of BRT trivializations. Note that, on BRT trivialization $(D,(z,\theta),\varphi)$, we have 
\begin{equation}\label{e-gue150514f}
\ddbar_b=\sum^n_{j=1}d\ol z_j\wedge(\frac{\partial}{\partial\ol z_j}-i\frac{\partial\varphi}{\partial\ol z_j}(z)\frac{\partial}{\partial\theta}). \nonumber
\end{equation}


\subsection{Local heat kernels on BRT trivializations}\label{SS:lochk}

Until further notice, we fix $m\in\mathbb Z$. 
Let $B:=(D,(z,\theta),\varphi)$ be a BRT trivialization. We may assume that $D=U\times]-\varepsilon,\varepsilon[$, where $\varepsilon>0$ and $U$ is an open set of $\Complex^n$. Since $E$ is rigid, we can consider $E$ as a holomorphic vector bundle over $U$. We may assume that $E$ is trivial on $U$. Consider $L\To U$ be a trivial line bundle with non-trivial Hermitian fiber metric $\abs{1}^2_{h^L}=e^{-2\varphi}$. Let $(L^m,h^{L^m})\To U$ be the $m$-th power of $(L,h^L)$. For every $q=0,1,2,\ldots,n$, let $\Omega^{0,q}(U,E\otimes L^m)$ and $\Omega^{0,q}(U,E)$ be the spaces of $(0,q)$ forms on $U$ with values in $E\otimes L^m$ and $E$, respectively. Put 
\[
\begin{split}
&\Omega^{0,\bullet}(U,E\otimes L^m):=\oplus_{j\in\set{0,1,\ldots,n}}\Omega^{0,j}(U,E\otimes L^m), \nonumber \\
&\Omega^{0,\bullet}(U,E):=\oplus_{j\in\set{0,1,\ldots,n}}\Omega^{0,j}(U,E). \nonumber
\end{split}\]
Since $L$ is trivial, from now on, we identify $\Omega^{0,\bullet}(U,E)$ with $\Omega^{0,\bullet}(U,E\otimes L^m)$.
Since the Hermitian fiber metric $\langle\,\cdot\,|\,\cdot\,\rangle_E$ is rigid, we can consider $\langle\,\cdot\,|\,\cdot\,\rangle_E$ as a Hermitian fiber metric on the holomorphic vector bundle $E$ over $U$. 
Let $\langle\,\cdot\,,\,\cdot\,\rangle$ be the Hermitian metric on $\Complex TU$ given by 
\[
\langle\,\frac{\pr}{\pr z_j}\,,\,\frac{\pr}{\pr z_k}\,\rangle=\langle\,\frac{\pr}{\pr z_j}+i\frac{\pr\varphi}{\pr z_j}(z)\frac{\pr}{\pr\theta}\,|\,\frac{\pr}{\pr z_k}+i\frac{\pr\varphi}{\pr z_k}(z)\frac{\pr}{\pr\theta}\,\rangle,\ \ j,k=1,2,\ldots,n.
\]
Then $\langle\,\cdot\,,\,\cdot\,\rangle$ induces a Hermitian metric on $T^{*0,\bullet}U:=\oplus_{j=0}^nT^{*0,j}U$, where  $T^{*0,j}U$ is the bundle of $(0,j)$ forms on $U$, $j=0,1,\ldots,n$. We shall also denote the Hermitian metric by $\langle\,\cdot\,,\,\cdot\,\rangle$. The Hermitian metrics on $T^{*0,\bullet}U$ and $E$ induce a
Hermitian metric on $T^{*0,\bullet}U\otimes E$. We shall also denote this induced metric by $\langle\,\cdot\,|\,\cdot\,\rangle_{E}$.
Let $(\,\cdot\,,\,\cdot\,)$ be the $L^2$ inner product on $\Omega^{0,\bullet}(U,E)$ induced by $\langle\,\cdot\,,\,\cdot\,\rangle$ and $\langle\,\cdot\,|\,\cdot\,\rangle_E$. Similarly, let $(\,\cdot\,,\,\cdot\,)_m$ be the $L^2$ inner product on $\Omega^{0,\bullet}(U,E\otimes L^m)$ induced by $\langle\,\cdot\,,\,\cdot\,\rangle$, $\langle\,\cdot\,|\,\cdot\,\rangle_E$ and $h^{L^m}$. 

Let
\[
\ddbar:\Omega^{0,\bullet}(U,E\otimes L^m)\To\Omega^{0,\bullet}(U,E\otimes L^m)
\]
be the Cauchy-Riemann operator and let
\[
\ol{\pr}^{*,m}:\Omega^{0,\bullet}(U,E\otimes L^m)\To\Omega^{0,\bullet}(U,E\otimes L^m)
\] 
be the formal adjoint of $\ddbar$ with respect to $(\,\cdot\,,\,\cdot\,)_m$. Put 
\begin{equation}\label{e-gue150606II}
D_{B,m} \, := \, \ddbar+\ol{\pr}^{*,m}: \Omega^{0,\bullet}(U,E\otimes L^m)\To\Omega^{0,\bullet}(U,E\otimes L^m). \nonumber
\end{equation}
Let 
\begin{equation}\label{e-gue150606II-2}
D^*_{B,m} \, : \, \Omega^{0,\bullet}(U,E\otimes L^m)\To\Omega^{0,\bullet}(U,E\otimes L^m) \nonumber
\end{equation}
be the formal adjoint of $D_{B,m}$ with respect to $(\, \cdot \, , \,  \cdot \, )_m$. Then we denote by
\begin{equation}
\Box_{B,m} \, = \, D_{B,m}D^*_{B,m} :  \Omega^{0,\bullet}(U,E\otimes L^m)\To\Omega^{0,\bullet}(U,E\otimes L^m). \nonumber
\end{equation}
Put 
\begin{equation}\label{E:Dbm}
D_{b,m} \, := \, \ddbar_{b,m}+\ddbar^*_{b,m}: \Omega^{0,\bullet}_m(X,E)\To\Omega^{0,\bullet}_m(X,E). \nonumber
\end{equation}
Let 
\begin{equation}\label{E:Dbm2}
D^*_{b,m} \, : \, \Omega^{0,\bullet}_m(X,E)\To\Omega^{0,\bullet}_m(X,E). \nonumber
\end{equation}
be the formal adjoint of $D_{b,m}$ with respect to $(\, \cdot \, | \, \cdot \, )_E$.
Denote by 
\begin{equation}
\ast \, : \, \Omega^{p,q}(U) \to \Omega^{n-p, n-q}(U) \nonumber
\end{equation}
the Hodge $\ast$-operator associated to the Riemannian metric $\langle \, \cdot \, , \, \cdot \, \rangle$ and $\ast_s$ the Hodge $\ast$-operators associated to the Riemannian metrics $\langle \, \cdot \, , \, \cdot \, \rangle_{s}, s \in [0, 1]$. 
Let $\mu^{E }_s$ be the induced conjugate linear bundle isomorphism of $E$ and $E^*$ associated to the metrics $\langle \, \cdot \, | \, \cdot \, \rangle_{E, s}$ and $\mu^{L^m }$ be the induced conjugate linear bundle isomorphism of $L^m$ and $(L^m)^*$ associated to the metric $h^{L^m}$.
Set 
\begin{equation}\label{E:qs}
Q_{s} = - \left(\ast_{s} \otimes \mu^{E}_s \otimes \mu^{L^m} \right)^{-1} \frac{\partial \left(\ast_{s} \otimes \mu^{E}_s \otimes \mu^{L^m} \right)}{\partial s} 
\end{equation}

\subsection{Local heat kernels depending on parameters}\label{SS:lochkpara}

Let $F$ be a vector bundle over a compact manifold $Z$. Let $\vartheta_1, \cdots, \vartheta_i$ be auxiliary Grassmann variables. We assume that the multiplication of any $q+1$ variables of the above given Grassmann variables vanishes, where $q$ is some fixed integer. Let $R(\vartheta_1, \cdots, \vartheta_i)$ be the Grassmann algebra generated by $1, \vartheta_1, \cdots, \vartheta_i$, cf. \cite{BF} or \cite[Subsection 3.1]{LY}. If $\omega \in R(\vartheta_1, \cdots, \vartheta_i)$, then $\omega$ is a linear combination of $\vartheta_{i_1}, \cdots, \vartheta_{i_k}$, where $1 \le i_1 < \cdots < i_k \le i$. We say that the monomial $\vartheta_{i_1}\cdots \vartheta_{i_k}$ is of degree $k$. Clearly, $k \le q$. We define elements of $\Omega^\bullet(Z) \otimes F$ to be of degree zero and give every monomial of $\Omega^\bullet(Z) \otimes \operatorname{End}(F) \widehat{\otimes} R(\vartheta_1, \cdots, \vartheta_i),$ say $\phi_{i_1 \cdots i_k}\vartheta_{i_1}\cdots \vartheta_{i_k}$, where $\phi_{i_1 \cdots i_k} \in \Omega^\bullet(Z) \otimes \operatorname{End}(F)$, a natural degree. Let $\left( \mathcal{B}, \| \cdot \| \right)$ be a normed space. We now introduce a norm on $\| \cdot \|_{\mathcal{B}\otimes R}$ on $\mathcal{B} \otimes R(\vartheta_1, \cdots, \vartheta_i)$ as follows. For $1 \le i_1 < \cdots < i_k \le i $,
\[
\phi \,  = \, \sum_{1\le k \le q}\phi_{i_1 \cdots i_k} \vartheta_{i_1} \cdots \vartheta_{i_k} \in \mathcal{B} \otimes R(\vartheta_{1}, \cdots, \vartheta_{i}),
\]
we define
\begin{equation}\label{E:norm}
\| \phi \|_{\mathcal{B}\otimes R} \, = \, \displaystyle{\operatorname{max}}_{_{1\le k \le q}} \| \phi_{i_1 \cdots i_k} \|.
\end{equation}


Now let $da, d\bar{a}$ be two odd Grassmann variables. Let $\eta \in \wedge^\bullet \mathbb{C}T^*U \widehat{\otimes} \mathbb{C}(da, d\bar{a})$, then $\eta$ can be written in the form
\begin{equation}
\eta \, = \, \eta_0 + da \eta_1 + d\bar{a} \eta_2 + dad\bar{a} \eta_3, \qquad \text{where} \ \eta_i \in \wedge^\bullet \mathbb{C}T^*U, \  0 \le i \le 3 \nonumber
\end{equation} 
and  we set 
\begin{equation}
(\eta)^{dad\bar{a}} \, = \, \eta_3. \nonumber
\end{equation} 
We will also identify $\eta$ as an element in $\wedge^\bullet \mathbb{C}T^*D \widehat{\otimes} \mathbb{C}(da, d\bar{a})$ naturally. We denote by
\begin{equation}
L_{b,m,t} \, = \, t \Box_{b,m} + \sqrt{\frac{t}{2}} da D_{b,m} + \sqrt{\frac{t}{2}} d\bar{a}[D_{b,m}, Q_b] -dad\bar{a}Q_b, \nonumber
\end{equation}
where $Q_b \, := \, Q_{b, 0}$ as defined in \eqref{E:qbs}.
Proceeding formally as in \cite[Theorem 1.20]{BGS3}, we obtain
\begin{proposition}\label{P:strqbm}
The following identity holds:
\begin{equation}
\frac{\partial}{\partial t} \left(\, t \operatorname{STr}  \left[ \, Q  \exp \left(\,-t\Box_{b,m} \, \right) \, \right] \, \right) \, =\, \operatorname{STr}  \left[ \, \exp \left( \, -L_{b,m,t} \, \right) \, \right]^{dad\bar{a}}. \nonumber
\end{equation}
\end{proposition}
We then set
\begin{equation}
L_{B,m,t} \, = \, t \Box_{B,m} + \sqrt{\frac{t}{2}} da D_{B,m} + \sqrt{\frac{t}{2}} d\bar{a}[D_{B,m}, Q] -dad\bar{a}Q, \nonumber
\end{equation}
where $Q \, := \, Q_0$ as defined in \eqref{E:qs}.

We have the following result (see also Lemma 5.1 in~\cite{CHT}). 

\begin{lemma}\label{l-gue150606}
Let $u\in\Omega^{0,\bullet}_m(X,E)$. On $D$, we write $u(z,\theta)=e^{im\theta}\Td u(z)$, $\Td u(z)\in\Omega^{0,\bullet}(U,E)$. Then,
\begin{equation}\label{e-gue150606III}
e^{-m\varphi}L_{B,m, t}(e^{m\varphi}\Td u)=e^{-im\theta}L_{b,m,t}(u). \nonumber
\end{equation}
\end{lemma}

Let $z, w\in U$ and let 
$$
T(z,w)\in \left((T^{*0,\bullet}_wU\otimes E_w)^*\boxtimes(T^{*0,\bullet}_zU\otimes E_z)\right) \widehat{\otimes} \mathbb{C}(da, d\bar{a}).
$$ 
We write $\abs{T(z,w)}$ to denote the standard pointwise matrix norm of $T(z,w)$ induced by $\langle\,\cdot\,|\,\cdot\,\rangle_E$ as in \eqref{E:norm}. Let $\Omega^{0,\bullet}_0(U,E)$ be the subspace of $\Omega^{0,\bullet}(U,E)$ whose elements have compact support in $U$. Let $dv_U$ be the volume form on $U$ induced by $\langle\,\cdot\,,\,\cdot\,\rangle$. Assume 
$$
T(z,w)\in C^\infty(U\times U,(T^{*0,\bullet}_wU\otimes E_w)^*\boxtimes(T^{*0,\bullet}_zU\otimes E_z))\widehat{\otimes} \mathbb{C}(da, d\bar{a}).
$$
Let $u\in\Omega^{0,\bullet}_0(U,E)$. We define the integral 
$$
\int T(z,w)u(w)dv_U(w)
$$ 
in the standard way. For any $t>0$, let 
$$
G(\hat{t},t,z,w)\in C^\infty(\Real_+\times U\times U,(T^{*0,\bullet}_wU\otimes E_w)^*\boxtimes(T^{*0,\bullet}_zU\otimes E_z))\widehat{\otimes} \mathbb{C}(da, d\bar{a}).
$$ 
For any $t>0$, we write $G(\,\hat{t}\,)$ to denote the continuous operator
\[
\begin{split}
G(\, \hat{t} \,):\Omega^{0,\bullet}_0(U,E)\widehat{\otimes} \mathbb{C}(da, d\bar{a})&\To\Omega^{0,\bullet}(U,E)\widehat{\otimes} \mathbb{C}(da, d\bar{a}), \\
u&\To\int G(\hat{t},t,z,w)u(w)dv_U(w)\end{split}
\]
and we write
$G'(\, \hat{t} \,)$ to denote the continuous operator
\[
\begin{split}
G'(\, \hat{t} \,):\Omega^{0,\bullet}_0(U,E)\widehat{\otimes} \mathbb{C}(da, d\bar{a})&\To\Omega^{0,\bullet}(U,E)\widehat{\otimes} \mathbb{C}(da, d\bar{a}),\\
u&\To\int \frac{\pr G(\hat{t},t,z,w)}{\pr \hat{t}}u(w)dv_U(w).\end{split}
\]

We have the following theorem, cf. \cite[Section 3]{LY},
\begin{theorem}\label{t-gue150607}
For any $t>0$, there is 
$$
A_{B,m}(\hat{t},t,z,w)\in C^\infty(\Real_+\times U\times U, (T^{*0,\bullet}_wU\otimes E_w)^*\boxtimes(T^{*0,\bullet}_zU\otimes E_z))\widehat{\otimes} \mathbb{C}(da, d\bar{a})
$$ 
such that 
\[
\begin{split}
&\mbox{$\lim_{\hat{t} \To0+}A_{B,m}(\hat{t})=I$ in $D'(U,T^{*0,\bullet}U\otimes E)\widehat{\otimes} \mathbb{C}(da, d\bar{a})$},\\
&A'_{B,m}(\hat{t})u+A_{B,m}(\hat{t})(L_{B,m,t}u)=0,\ \ \forall u\in\Omega^{0,\bullet}_0(U,E),\ \ \forall \hat{t}>0,
\end{split}
\]
and $A_{B,m}(\hat{t}, t,z,w)$ satisfies the following: 
(I)  For every small $t>0$, every compact set $K\Subset U$ and every $\alpha_1, \alpha_2, \beta_1, \beta_2\in\mathbb N^n_0$, every $\gamma \in \mathbb{N}_0$,  there are constants $C_{\gamma, \alpha_1,\alpha_2,\beta_1,\beta_2,K}>0$, $\varepsilon_0>0$, $P \in \mathbb{N}$ independent of $t$such that
\begin{equation}\label{e-gue160128w}
\begin{split}
&|\pr^{\gamma}_{\hat{t}}\pr^{\alpha_1}_z\pr^{\alpha_2}_{\ol z}\pr^{\beta_1}_w\pr^{\beta_2}_{\ol w} A_{B,m}(1,t,z,w) |\\
&\leq C_{\alpha_1,\alpha_2,\beta_1,\beta_2,K}t^{-P}e^{-\varepsilon_0\frac{\abs{z-w}^2}{t}},\ \  \forall (t, z,w)\in \Real_+\times K\times K.
\end{split}
\end{equation}
(II) 
$A_{B,m}(1,t,z,w)$ admits an asymptotic expansion:
\begin{eqnarray}\label{E:asymptotic}
&& A_{B,m}(1,t,z,w)  =  e^{-\frac{h(z,w)}{t}}K_{B,m}(1,t,z,w),   \\
&& K_{B,m}(1,t,z,w)  \sim t^{-1}a_{-1}(z,w) + a_0(z,w) + ta_{1}(z,w) +\cdots \ \text{as} \ t \to 0^+, \nonumber \\
&& a_j(z,w) \in C^\infty \left(\, U \times U , (T_w^{*0,\bullet}U \otimes E_w)^* \otimes (T_z^{\ast0,\bullet}U \otimes E_z) \, \right) \widehat{\otimes} \mathbb{C}(da, d\bar{a}), \quad j=-1, 0, 1, \cdots, \nonumber
\end{eqnarray}
where $h(z,w)\in C^\infty(U\times U)$ and for every compact set $K\Subset U$, there is a constant $C>1$ such that $\frac{1}{C}\abs{z-w}^2\leq h(z,w)\leq C\abs{z-w}^2$, for all $(z,w)\in K\times K$. 
\end{theorem} 

Assume that $X=D_1\bigcup D_2\bigcup\cdots\bigcup D_N$, where $B_j:=(D_j,(z,\theta),\varphi_j)$ is a BRT trivialization, for each $j$. We may assume that for each $j$, 
$$
D_j=U_j\times]-2\delta_j,2\Td\delta_j[\subset\Complex^n\times\Real, \quad \delta_j>0, \quad \Td\delta_j>0, \quad U_j=\set{z\in\Complex^n;\, \abs{z}<\gamma_j}.
$$ 
For each $j$, put 
$$
\hat D_j=\hat U_j\times]-\frac{\delta_j}{2},\frac{\Td\delta_j}{2}[,
$$ 
where $\hat U_j=\set{z\in\Complex^n;\, \abs{z}<\frac{\gamma_j}{2}}$. We may suppose that 
$$
X=\hat D_1\bigcup\hat D_2\bigcup\cdots\bigcup\hat D_N.
$$ 
Let $\chi_j\in C^\infty_0(\hat D_j)$, $j=1,2,\ldots,N$, with $\sum^N_{j=1}\chi_j=1$ on $X$. Fix $j=1,2,\ldots,N$. Put 
\[K_j=\set{z\in\hat U_j;\, \mbox{there is a $\theta\in]-\frac{\delta_j}{2},\frac{\Td\delta_j}{2}[$ such that $\chi_j(z,\theta)\neq0$}}.\]
Let $\tau_j(z)\in C^\infty_0(\hat U_j)$ with $\tau_j\equiv1$ on some neighborhood $W_j$ of $K_j$. Let $\sigma_j\in C^\infty_0(]-\frac{\delta_j}{2},\frac{\Td\delta_j}{2}[)$ with $\int\sigma_j(\theta)d\theta=1$. For any $t>0$, let 
$$
A_{B_j,m}(\hat{t}, t,z,w)\in C^\infty(\Real_+\times U_j\times U_j,(T^{*0,\bullet}_wU_j\otimes E_w)^*\boxtimes(T^{*0,\bullet}_zU_j\otimes E_z))
$$ 
be as in Theorem~\ref{t-gue150607}. For any $t>0$, put 
\begin{equation}\label{e-gue150627f}
H_{j,m}(\hat{t}, t,x,y)=\chi_j(x)e^{-m\varphi_j(z)+im\theta}A_{B_j,m}(\hat{t}, t,z,w)e^{m\varphi_j(w)-im\eta}\tau_j(w)\sigma_j(\eta),
\end{equation}
where $x=(z,\theta)$, $y=(w,\eta)\in\Complex^{n}\times\Real$. Let 
\begin{equation}\label{e-gue150626fIII}
\begin{split}
\Gamma_m(\hat{t}, t,x,y):=\frac{1}{2\pi}\sum^N_{j=1}\int^\pi_{-\pi}H_{j,m}(\hat{t}, t,x,e^{iu}\circ y)e^{imu}du.
\end{split}
\end{equation}
From Lemma~\ref{l-gue150606}, off-diagonal estimates of $A_{B_j,m}(\hat{t}, t,x,y)$ (see \eqref{e-gue160128w} and \eqref{E:asymptotic}), we can repeat the proof of Theorem 5.11 in~\cite{CHT} with minor change and deduce that 

\begin{theorem}\label{t-gue150630I}
For every $\ell\in\mathbb N$, $\ell\geq2$, and every $\epsilon>0$, there are $\epsilon_0>0$ independent of $t$ such that
\[
\norm{e^{-L_{b,m, t}}(x,y)-\Gamma_m(1, t,x,y)}_{C^\ell(X\times X)}\leq e^{-\frac{\epsilon_0}{t}},\ \ \forall t\in(0,\epsilon).
\]
\end{theorem}


\subsection{Constant term of the heat kernel asymptotics of the modified Kohn Laplacians on BRT trivializations}\label{SS:const}

To state the result precisely, we introduce some notations. Let $\nabla^{TU}$ and $\nabla'^{TU}$be the Levi-Civita connections on $\mathbb{C}TU$ with respect to the metrics $\langle \cdot, \cdot \rangle$ and $\langle \cdot, \cdot \rangle'$, repsectively. Let $P_{T^{1,0}U}$ be the natural projection from $\mathbb{C}TU$ onto $T^{1,0}U.$ Then, 
$$
\nabla^{T^{1,0}U}:=P_{T^{1,0}U}\nabla^{TU} \quad \text{and} \quad \nabla'^{T^{1,0}U}:=P_{T^{1,0}U}\nabla'^{TU}
$$ 
are connections on $T^{1,0}U$. Let $\nabla^{E\otimes L^m}$ be the Chern connection on $E \otimes L^m \to U$ induced by the metrics $\langle \cdot, \cdot \rangle_E$ and $h^{L^m}$ and $\nabla'^{E\otimes L^m}$ be the Chern connection on $E \otimes L^m \to U$ induced by the metrics $\langle \cdot, \cdot \rangle'_E$ and $h^{L^m}$. Let 
$$
 \Theta^{T^{1,0}U}  \, := \,  \Theta(\nabla^{T^{1,0}U}, T^{1,0}U) \in C^\infty(U, \wedge^2( \mathbb{C}T^*U) \otimes \operatorname{End}(T^{1,0}U))
 $$
 and 
 $$
 \Theta^{E \otimes L^m}  \, := \, \Theta(\nabla^{E\otimes L^m}, E\otimes L^m) \in C^\infty(U, \wedge^2( \mathbb{C}T^*U) \otimes \operatorname{End}(E \otimes L^m))
 $$ 
 be the curvatures induced by $\nabla^{T^{1,0}U}$ and $\nabla^{E \otimes L^m}$, respectively. Similarly, let 
 $$
 \Theta'^{T^{1,0}U}  \in C^\infty(U, \wedge^2( \mathbb{C}T^*U) \otimes \operatorname{End}(T^{1,0}U))
 $$ 
 and 
 $$
 \Theta'^{E \otimes L^m} \in C^\infty(U, \wedge^2( \mathbb{C}T^*U) \otimes \operatorname{End}(E \otimes L^m))
 $$ 
 be the curvatures induced by $\nabla'^{T^{1,0}U}$ and $\nabla'^{E \otimes L^m}$, respectively. As in complex geometry, put
$$
\operatorname{Td} \left(  \Theta^{T^{1,0}U}   \right) \, := \,\operatorname{Td} \left( \nabla^{T^{1,0}U}, T^{1,0}U \right)  \, = \, e^{\operatorname{Tr}\left( h \left(  \frac{i}{2\pi} \Theta \left( \nabla^{T^{1,0}U}, T^{1,0}U \right) \right)  \right) },
$$
where 
$$
h \left( z \right) \, = \, \log \left( \frac{z}{1-e^{-z}} \right),
$$
and
$$
\operatorname{ch}\left( \Theta^{E \otimes L^m} \right) \, := \, \operatorname{ch} \left( \nabla^{E \otimes L^m}, E \otimes L^m \right) \, = \, \operatorname{Tr} \left( \widetilde{h} \left( \frac{i}{2\pi} \Theta \left( \nabla^{E\otimes L^m}, E\otimes L^m \right)  \right) \right),
$$
where 
$$
\widetilde{h} \left( z \right) \, = \, e^z.
$$
We also define $\operatorname{Td} \left(  \Theta'^{T^{1,0}U}   \right)$ and $\operatorname{ch}\left( \Theta'^{E \otimes L^m} \right)$ in similar ways.

Let $g^{U}, g'^{U}$ be two Hermitian metrics on $T^{1,0}U$ and $h^{E}, h'^{E}$ be two Hermitian metrics on $E$. Consider a smooth family of metrics $s \in [0,1] \to (g^U_s, h^E_s)$ on $T^{1,0}U$ and $E$ such that
\[
(g_0^U, h_0^{E}) \, = \, (g^U, h^{E}) \quad \text{and} \quad (g_1^U, h_1^{E}) \, = \, (g'^U, h'^{E}).
\]
Let $\nabla_s^{T^{1,0}U}$ and $\nabla_s^{E \otimes L^m}$ be the connections on $T^{1,0}U$ and on $E \otimes L^m \to U$ induced metrics $g^{U}_s$ and $(h^E_s, h^{L^m})$, respectively,
such that 
\[
(\nabla_0^{T^{1,0}U}, \nabla_1^{T^{1,0}U}) \, = \, (\nabla^{T^{1,0}U}, \nabla'^{T^{1,0}U})
\]
 and 
 \[ 
 (\nabla_0^{E \otimes L^m}, \nabla_1^{E \otimes L^m}) \, =\, (\nabla^{E \otimes L^m}, \nabla'^{E \otimes L^m}).
\]
Let 
$$
 \Theta^{T^{1,0}U}_s  \, := \,  \Theta(\nabla_s^{T^{1,0}U}, T^{1,0}U) \in C^\infty(U, \wedge^2( \mathbb{C}T^*U) \otimes \operatorname{End}(T^{1,0}U))
 $$ 
 and 
 $$
  \Theta_s^{E \otimes L^m}  \, := \, \Theta (\nabla_s^{E\otimes L^m}, E\otimes L^m) \in C^\infty(U, \wedge^2( \mathbb{C}T^*U) \otimes \operatorname{End}(E \otimes L^m))
 $$ 
 be the curvatures induced by $\nabla_s^{T^{1,0}U}$ and $\nabla_s^{E \otimes L^m}$, respectively, such that 
 $$
 (\Theta_0^{T^{1,0}U}, \Theta_1^{T^{1,0}U}) \, = \,  (\Theta^{T^{1,0}U}, \Theta'^{T^{1,0}U}) 
 $$
 and 
 $$ 
  (\Theta_0^{E \otimes L^m}, \Theta_1^{E \otimes L^m}) \, = \,  (\Theta^{E \otimes L^m}, \Theta'^{E \otimes L^m}).
 $$ 
Let 
$$
P^U \, = \, \oplus^n_{p=0}\Omega^{p,p}(U).
$$ 
Let $P'^U \subset P^U$ be the set of smooth forms $\alpha \in P^U$ such that there exist smooth forms $\beta, \gamma$ on $X$ for which 
$$
\alpha \, = \, \partial \beta + \overline{\partial} \gamma.
$$ 
By the results of \cite[\S (e)]{BGS1}, the form
\begin{eqnarray}
\alpha & := & (2\pi \sqrt{-1})^{-n} \int_0^1 \frac{\partial}{\partial v}\Big|_{v=0} \Big\{  \operatorname{Td}\Big( -\Theta^{T^{1,0}U}_s - v ( g_s^{U})^{-1} \frac{\partial g^U_s}{\partial s} \Big)  \nonumber  \\
& & \wedge  \operatorname{Tr} \left( \exp \left( -\Theta^{E \otimes L^m}_s - v  ( h_s^{E \otimes L^m} )^{-1} \frac{\partial h_s^{E \otimes L^m}}{\partial s}\right) \right) \, \Big\}dv \nonumber
\end{eqnarray} 
defines an element in $P^U/P'^U$ which depends only on $(g^U, h^{E \otimes L^m})$ and $ (g'^U, h'^{E \otimes L^m})$. 
Since the Hermitian metric $h^L$ does not depend on the parameter $s$, we can easily see that
\begin{eqnarray}\label{E:formalpha}
\alpha & = & (2\pi \sqrt{-1})^{-n} \int_0^1 \frac{\partial}{\partial v}\Big|_{v=0} \Big\{  \operatorname{Td}\Big( -\Theta^{T^{1,0}U}_s - v ( g_s^{U})^{-1} \frac{\partial g^U_s}{\partial s} \Big)  \nonumber  \\
& & \wedge  \operatorname{Tr} \left( \exp \left( -\Theta^{E}_s - v  ( h_s^{E} )^{-1} \frac{\partial h_s^{E}}{\partial s}\right) \right) \, \Big\}dv  \wedge \operatorname{Tr} \left( \exp \left( -\Theta^{L^m} \right) \right) 
\end{eqnarray} 
Recall that, cf. \cite[(1.103), (1.136), (1.138)]{BGS3},
\begin{equation}\label{E:strabm}
\lim_{t \to 0} \operatorname{STr} A_{B,m}(1,t,z,w) ^{dad\bar{a}} \, = \, \alpha \, = \, \operatorname{STr} a_0(z,z). \nonumber
\end{equation}

By \cite[\S (f)]{BGS1}, there is uniquely defined secondary characteristic (Bott-Chern) classes 
$$
\widetilde{\operatorname{Td}}(\nabla^{T^{1,0}U}, \nabla'^{T^{1,0}U}, T^{1,0}U)
$$ 
and 
$$
\widetilde{\operatorname{ch}}(\nabla^{E \otimes L^m}, \nabla'^{E \otimes L^m}, E \otimes L^m)
$$ 
in $P^U/P'^U$ such that
\begin{eqnarray}\label{E:bcchern}
\frac{\overline{\partial}\partial}{2\pi\sqrt{-1}} \widetilde{\operatorname{Td}}(\nabla^{T^{1,0}U}, \nabla'^{T^{1,0}U}, T^{1,0}U) \, = \, \operatorname{Td}(\nabla'^{T^{1,0}U}, T^{1,0}U) - \operatorname{Td}(\nabla^{T^{1,0}U}, T^{1,0}U), \nonumber \\
\frac{\overline{\partial}\partial}{2\pi\sqrt{-1}} \widetilde{\operatorname{ch}}(\nabla^{E \otimes L^m}, \nabla'^{E \otimes L^m}, E \otimes L^m) \, = \, \operatorname{ch}(\nabla'^{E \otimes L^m}, E \otimes L^m) - \operatorname{ch}(\nabla^{E \otimes L^m}, E \otimes L^m).
\end{eqnarray}
Hence, we have 
\begin{equation}\label{E:stra00}
\operatorname{STr} a_0(z,z)dv_U(z) \, = \, [\alpha^{n,n}](z), \quad \forall z \in U. \nonumber
\end{equation}

According to \cite[Theorem 1.27, 1.29 and Corollary 1.30]{BGS1}, the component of degree $(n,n)$ of $\alpha$ represents in $P^U/P'^U$ the corresponding component of the Bott-Chern class
\begin{equation}\label{E:alphann}
\alpha^{n,n} \, := \, \widetilde{\operatorname{Td}}(\nabla^{T^{1,0}U}, \nabla'^{T^{1,0}U}, T^{1,0}U)\wedge \operatorname{ch}(\nabla^{E \otimes L^m}) + \operatorname{Td}(\nabla^{T^{1,0}U}) \wedge \widetilde{\operatorname{ch}}(\nabla^E\otimes L^m, \nabla'^{E \otimes L^m}, E \otimes L^m). 
\end{equation} 
 
By \eqref{E:formalpha}, \eqref{E:bcchern} and \eqref{E:alphann}, we have
\begin{equation}\label{E:alphann1}
\alpha^{n,n} \, = \, \left( \widetilde{\operatorname{Td}}(\nabla^{T^{1,0}U}, \nabla'^{T^{1,0}U}, T^{1,0}U)\wedge \operatorname{ch}(\nabla^{E}) + \operatorname{Td}(\nabla^{T^{1,0}U}) \wedge \widetilde{\operatorname{ch}}(\nabla, \nabla'^{E}, E) \right)\wedge \operatorname{ch}(\nabla^{L^m}, L^m). \nonumber
\end{equation}

Let $\nabla^{TX}$ and $\nabla'^{TX}$ be the Levit-Civita connections on $TX$ with respect to $\langle \, \cdot\,, \, \cdot \, \rangle$ and $\langle \, \cdot \, , \, \cdot \, \rangle'$, respectively. Let $P_{T^{1,0}X}$ be the natural projection from $\mathbb{C}TX$ onto $T^{1,0}X$. Then, 
$$
\nabla^{T^{1,0}X}\, := \, P_{T^{1,0}X}\nabla^{TX}
$$
and 
$$
 \nabla'^{T^{1,0}X}\, := \, P_{T^{1,0}X}\nabla'^{TX}
$$
are connections on $T^{1,0}X$. Let $\nabla^E$ and $\nabla'^E$ be the connections on $E$ induced by $\langle \, \cdot \, | \, \cdot \, \rangle_E$ and $\langle \, \cdot \, | \, \cdot \, \rangle'_E$, respectively. We can check that $\nabla^{T^{1,0}X}, \nabla'^{T^{1,0}X}, \nabla'^E$ and $\nabla'^E$ are rigid. Moreover, it is straightforward to check that
\begin{eqnarray}
\operatorname{Td} \left( \nabla^{T^{1,0}U}, T^{1,0}U \right) (z)  \, = \, \operatorname{Td}_b \left( \nabla^{T^{1,0}X}, T^{1,0}X \right) (z, \theta), \quad \forall (z, \theta) \, \in \, D, \nonumber \\
\operatorname{ch} \left( \nabla^{E \otimes L^m}, E\otimes L^m \right)(z) \, = \, \left( \operatorname{ch}_b (\nabla^E, E) \wedge e^{-m\frac{d\omega_0}{2\pi}}  \right)(z, \theta), \quad \forall (z, \theta) \, \in \, D. \nonumber 
\end{eqnarray}

We can check that
\begin{eqnarray}
\widetilde{\operatorname{Td}} \left( \nabla^{T^{1,0}U}, \nabla'^{T^{1,0}U}, T^{1,0}U \right) (z)  \, = \, \widetilde{\operatorname{Td}}_b \left( \nabla^{T^{1,0}X}, \nabla'^{T^{1,0}X}, T^{1,0}X \right) (z, \theta), \quad \forall (z, \theta) \, \in \, D, \nonumber \\
\widetilde{\operatorname{ch}} \left( \nabla^{E \otimes L^m}, \nabla'^{E \otimes L^m}, E\otimes L^m \right)(z) \, = \, \left( \widetilde{\operatorname{ch}}_b (\nabla^{E}, \nabla'^{E}, E) \wedge e^{-m\frac{d\omega_0}{2\pi}} \right)(z, \theta), \quad \forall (z, \theta) \, \in \, D.\nonumber 
\end{eqnarray}
and
\begin{eqnarray}\label{E:formequal}
&& \left[ \widetilde{\operatorname{Td}} \left( \nabla^{T^{1,0}U}, \nabla'^{T^{1,0}U}, T^{1,0}U \right) \wedge \operatorname{ch} \left( \nabla^{E \otimes L^m}, E\otimes L^m \right) \right]_{2n}(z) \wedge d\theta \nonumber \\
&& = \left[ \widetilde{\operatorname{Td}}_b \left( \nabla^{T^{1,0}X}, \nabla'^{T^{1,0}X}, T^{1,0}X \right) \wedge \operatorname{ch}_b \left( \nabla^{E}, E \right) \wedge e^{-m\frac{d\omega_0}{2\pi}} \wedge \omega_0 \right]_{2n+1}(z, \theta), \quad \forall (z, \theta) \in D, \nonumber \\
&& \left[ \operatorname{Td} \left( \nabla^{T^{1,0}U}, T^{1,0}U \right) \wedge \widetilde{\operatorname{ch}} \left( \nabla^{E \otimes L^m},  \nabla'^{E \otimes L^m}, E\otimes L^m \right) \right]_{2n}(z) \wedge d\theta \nonumber \\
 && = \left[ \operatorname{Td}_b \left( \nabla^{T^{1,0}X}, T^{1,0}X \right) \wedge \widetilde{\operatorname{ch}}_b \left( \nabla^{E}, \nabla'^{E}, E \right) \wedge e^{-m\frac{d\omega_0}{2\pi}} \wedge \omega_0 \right]_{2n+1}(z, \theta), \quad \forall (z, \theta) \in D, \nonumber 
\end{eqnarray}
where 
$\left[ \widetilde{\operatorname{Td}}_b \left( \nabla^{T^{1,0}X}, \nabla'^{T^{1,0}X}, T^{1,0}X \right) \wedge \operatorname{ch}_b \left( \nabla^{E}, E \right) \wedge e^{-m\frac{d\omega_0}{2\pi}} \wedge \omega_0 \right]_{2n+1}$
denotes the $2n+1$ forms part of 
\[
\widetilde{\operatorname{Td}}_b \left( \nabla^{T^{1,0}X}, \nabla'^{T^{1,0}X}, T^{1,0}X \right) \wedge \operatorname{ch}_b \left( \nabla^{E}, E \right) \wedge e^{-m\frac{d\omega_0}{2\pi}} \wedge \omega_0 
\]
and $ \left[ \operatorname{Td}_b \left( \nabla^{T^{1,0}X}, T^{1,0}X \right) \wedge \widetilde{\operatorname{ch}}_b \left( \nabla^{E}, \nabla'^{E}, E \right) \wedge e^{-m \frac{d\omega_0}{2\pi}} \wedge \omega_0 \right]_{2n+1}$ denotes the $2n+1$ parts of
\[
\operatorname{Td}_b \left( \nabla^{T^{1,0}X}, T^{1,0}X \right) \wedge \widetilde{\operatorname{ch}}_b \left( \nabla^{E}, \nabla'^{E}, E \right) \wedge e^{-m \frac{d\omega_0}{2\pi}}\wedge \omega_0.
\]
From the above equations and note that 
$$
dv_U \wedge d\theta = dv_X
$$ 
on $D$, we get
\begin{theorem}\label{T:a0zz}
With the notations above, we have, for all $(z, \theta) \in D$,
\begin{eqnarray}
\operatorname{STr} a_0(z,z) dv_X(z,\theta) & = & \left[ \widetilde{\operatorname{Td}}_b \left( \nabla^{T^{1,0}X}, \nabla'^{T^{1,0}X}, T^{1,0}X \right) \wedge \operatorname{ch}_b \left( \nabla^{E}, E \right) \wedge e^{-m\frac{d\omega_0}{2\pi}} \wedge \omega_0 \right]_{2n+1}(z, \theta) \nonumber \\
 & & +  \left[ \operatorname{Td}_b \left( \nabla^{T^{1,0}X}, T^{1,0}X \right) \wedge \widetilde{\operatorname{ch}}_b \left( \nabla^{E}, \nabla'^{E}, E \right) \wedge e^{-m\frac{d\omega_0}{2\pi}} \wedge \omega_0 \right]_{2n+1}(z, \theta). \nonumber
\end{eqnarray}
\end{theorem}


\subsection{Proof of Theorem \ref{T:main}}\label{SS:proof}

The theorem follows by combining Theorem \ref{T:5.5.6}, Proposition \ref{P:strqbm}, Lemma \ref{l-gue150606}, Theorem \ref{t-gue150607}, \eqref{e-gue150627f}, \eqref{e-gue150626fIII}, Theorem \ref{t-gue150630I}, \eqref{E:strabm} and Theorem \ref{T:a0zz}.


\section{The asymptotic anomaly formula of the analytic torsion}

In this section we will deduce an asymptotic anomaly formula for the $L^2$-metric on $\lambda_{b,m}(E)$. The formula is an CR analogue of Theorem 5.5.12 of \cite{MM}. 

\subsection{Asymptotic anomaly formula for the $L^2$-metric}

We now define the canonical line bundle $K_X$ of $(X, T^{1,0}X)$ by
\[
K_X \, = \, \wedge^{n}T^{*1,0}X.
\]
We denote by $K_X^*$ the dual of the canonical line bundle $K_X$ on $X$. Let $\langle\, \cdot \, | \, \cdot \, \rangle_0$ and $\langle \, \cdot \, | \, \cdot \, \rangle_1$ be two rigid Hermitian metrics on $\mathbb{C}TX$. We keep the rigid Hermtitian metric $h^E$ on $E$ fixed. Let $|\cdot|_{K_X^*,0}$ and $|\cdot|_{K_X^*,1}$ be the metrics on $K_X^*$ induced by the metrics $\langle\, \cdot \, | \, \cdot \, \rangle_0$ and $\langle \, \cdot \, | \, \cdot \, \rangle_1$, respectively. Let $|| \cdot ||_{m,0}$ and $||\cdot||_{m,1}$ be the Quillen metrics on $\lambda_{b,m}(E)$ induced by the metrics $\langle\, \cdot \, | \, \cdot \, \rangle_0$ and $\langle \, \cdot \, | \, \cdot \, \rangle_1$, respectively, and the given rigid Hermitian metric $h^E$ on $E$. Let $|\cdot|_{m,0}$ and $|\cdot|_{m,1}$ be the $L^2$ metrics on $\lambda_{b,m}(E)$ induced by the metrics $\langle\, \cdot \, | \, \cdot \, \rangle_0$ and $\langle \, \cdot \, | \, \cdot \, \rangle_1$, respectively, and the given rigid Hermitian metric $h^E$ on $E$.
\begin{theorem}\label{T:5.5.12}
As $m \to \infty$, we have
\begin{equation}\label{E:5.5.61}
\log \left( \frac{|\, \cdot \, |^2_{b,m,1}}{|\, \cdot\, |^2_{b,m,0}} \right) \, = \, - \operatorname{rk}(E) \int_X \log \left( \frac{|\, \cdot\, |^2_{K_X^*,1}}{|\, \cdot\, |^2_{K_X^*,0}} \right) e^{-m \frac{d\omega_0}{2\pi} } \wedge \omega_0 +o(m^n).
\end{equation}
\end{theorem}
\begin{proof}
Let $\theta_{b,m,0}(z), \theta_{b,m,1}(z)$ be the $\zeta$-functions, defined as in \eqref{E:5.5.12}, associated with the rigid Hermitian metrics $\langle\, \cdot \, | \, \cdot \, \rangle_0$ and $\langle \, \cdot \, | \, \cdot \, \rangle_1$, respectively, and with the given rigid Hermitian metric $h^E$ on $E$. By Theorem 1.1 of \cite{HH} and 
$$
| \sigma |^2_{K_{X^*,i}} \, = \, |\Theta^n_i(\sigma, \overline{\sigma})|/n!, \quad i=0,1,
$$ 
we have
\begin{equation}\label{E:5.5.62}
\theta'_{b,m,1}(0) - \theta'_{b,m,0}(0) = - \frac{\operatorname{rk}(E)}{2} \int_X \log \left( \frac{|\cdot|^2_{K_X^*,1}}{|\cdot|^2_{K_X^*,0}} \right) e^{-m \frac{d\omega_0}{2\pi} } \wedge \omega_0 +o(m^n).
\end{equation}
We next choose a path of metrics $g_s\, :=\, \langle \, \cdot \, | \, \cdot \, \rangle_s, s \in [0, 1],$ connecting $\langle\, \cdot \, | \, \cdot \, \rangle_0$ and $\langle \, \cdot \, | \, \cdot \, \rangle_1$. We denote the objects associated to $\langle \, \cdot \, | \, \cdot \, \rangle_s$ with a subscript $s$. Then, by Theorem \ref{T:5.5.6} and Theorem 3.6 of \cite{HH}, we have
\begin{equation}\label{E:5.5.63}
\frac{\partial}{\partial s} \log \left( ||\cdot||^2_{\lambda_{b,m,s}(E)} \right) \, =  \, m^n \int_X \operatorname{STr} \Big[ Q_sA_{0,s} \otimes \operatorname{Id}_E \Big] dv_{X,s} + O(m^{n-1/2}),
\end{equation}
where $A_{0}$ is defined in (3.34) of \cite{HH}.
By proceeding as in \cite[P. 261]{MM}, we get (cf. \cite[(5.5.68)]{MM})
\begin{equation}\label{E:5.5.68}
\int_X \operatorname{STr} \Big[ Q_sA_{0,s} \Big] dv_{X,s} \, =  \, -\frac{1}{2} \int_X \operatorname{Tr}\big|_{T^{1,0}X} \Big[ \frac{\partial g_s}{\partial s} \Big] e^{-\frac{d\omega_0}{2\pi}} \wedge \omega_0
 \end{equation}
 By \eqref{E:5.5.63} and \eqref{E:5.5.68}, we have
\begin{equation}\label{E:5.5.69}
\log \left( \frac{||\cdot||^2_{b,m,1}}{||\cdot||^2_{b,m,0}} \right) = - \frac{\operatorname{rk}(E)}{2} \int_X \log \left( \frac{|\cdot|^2_{K_X^*,1}}{|\cdot|^2_{K_X^*,0}} \right) e^{-m \frac{d\omega_0}{2\pi} } \wedge \omega_0 +O(m^{n-1/2}).
\end{equation}
Finally, by the definition of the Quillen metric (see Definition \ref{D:5.5.5}), we have
\begin{equation}\label{E:5.5.70}
\log \left( \frac{||\cdot||^2_{b,m,1}}{||\cdot||^2_{b,m,0}} \right) = \theta'_{b,m,1}(0) - \theta'_{b,m,0} +  \log \left( \frac{|\cdot|^2_{b,m,1}}{|\cdot|^2_{b,m,0}} \right).
 \end{equation}
By \eqref{E:5.5.62}, \eqref{E:5.5.69} and \eqref{E:5.5.70}, we get \eqref{E:5.5.61}.
\end{proof}


\bibliographystyle{plain}

\end{document}